\documentclass[mn,a4paper,fleqn%
]{w-art}
\usepackage{times,cite,w-thm}
\usepackage[PostScript]{diagrams}
\diagramstyle[labelstyle=\scriptstyle]
\def\CD{\diagram[amstex]}
\usepackage{amssymb}
\theoremstyle{plain}

\theoremstyle{definition}

\usepackage[]{graphicx}
\numberwithin{equation}{section}
\newcommand{\wA}{\widetilde A}

\newcommand{\grad}{\operatorname {grad}}
\newcommand{\divg}{\operatorname {div}}
\newcommand{\supp}{\operatorname{supp}}
\newcommand{\Ran}{\operatorname{ran}}
\newcommand{\pr}{\operatorname{pr}}
\newcommand{\ij}{\operatorname{i}}

\newcommand{\crp}{\overline{\mathbb R}_+}

\newcommand{\rnp}{{\mathbb R}^n_+}

\newcommand{\crnp}{\overline{\mathbb R}^n_+}

\newcommand{\comega}{\overline\Omega }
\newcommand{\ct}{\widetilde{c} }
\newcommand{\At}{\widetilde{A} }

\newcommand{\Op}{\operatorname{Op}}

\newcommand{\simto}{\overset\sim\rightarrow}

\newcommand{\rp}{ \mathbb R_+}

\newcommand{\Ami}{A_{\min}}
\newcommand{\Ama}{A_{\max}}
\def\C{\mathbb C}

\begin{document}
\DOIsuffix{theDOIsuffix}
\Volume{248}
\Month{01}
\Year{2007}
\pagespan{1}{}
\Receiveddate{XXXX}
\Reviseddate{XXXX}
\Accepteddate{XXXX}
\Dateposted{XXXX}



\title[ $M$-functions for closed extensions  ]{$M$-functions for closed extensions of adjoint pairs of operators with applications to elliptic boundary problems}


\author[B. M. Brown]{B. M. Brown\inst{1,}%
 \footnote{\textsf{malcolm@cs.cf.ac.uk},
            Phone: +44\,(0)29\,2087\,4812,
            Fax: +44\, (0)29\, 2087\, 4598}}
\address[\inst{1}]{Cardiff School of Computer Science, 
Cardiff University, 
Queen's Buildings, 
5 The Parade, Roath,
Cardiff CF24 3AA, 
UK}
\author[G. Grubb]{G. Grubb\inst{2,}\footnote{Corresponding author\quad E-mail:~\textsf{grubb@math.ku.dk}, Phone +45\, 3532\,0743  Fax  +45\, 3532\,0704 .}}
\address[\inst{2}]{Department of Mathematical Sciences,
University of Copenhagen, 
Universitetsparken 5,
DK-2100 K\o{}benhavn,  
Denmark}
\author[I. G. Wood]{I. G. Wood\inst{3,} \footnote{\textsf{ian.wood@aber.ac.uk},
            Phone +44\, (0)1970-62\,2758
Fax +44 (0)1970-62\,2826}}
\address[\inst{3}]{Institute of Mathematics and Physics,
Aberystwyth University,
Ceredigion SY23 3BZ,
Wales UK}
    \dedicatory{Dedicated to the memory of Leonid R.\ Volevi\v c{}.}
\begin{abstract}

In this paper, we combine results on extensions of operators
with recent results 
on the relation between the $M$-function and the spectrum, 
to examine the spectral behaviour of  boundary value problems.
$M$-functions are defined for general closed extensions, and
associated with realisations of elliptic operators. In
particular, we consider both ODE and PDE examples   
where it is possible for the operator to possess spectral points that can not be detected  by the $M$-function. 
\end{abstract}
\maketitle                   






\section{Introduction}
 
The extension theory for unbounded operators in Hilbert spaces has been studied since at least 1929 when von Neumann discovered the so-called 
Kre\u{\i}n extension. There are many applications of a general extension theory  
to problems generated by both ODE and PDE examples.  In the case of symmetric ODEs the book of Na\u{\i}mark \cite{N68}   characterises the extensions of the minimal operator by means of   a Cayley transform 
between the deficiency spaces  and determines all of these extensions by  the imposition of explicit   boundary conditions.
For PDEs, adjoint pairs of second order elliptic operators, their extensions and boundary value problems
 were studied in the paper of Vishik \cite{V52} while Grubb \cite{G68} showed  that all closed extensions 
 of the minimal operator can
be characterised by nonlocal boundary conditions,
 building on work of Lions and
Magenes  \cite{LM63,LM68} (cf.\ also H\"ormander \cite{H63}) for elliptic operators.

The theory of boundary value spaces (also known as boundary triplets) associated with symmetric operators
has its origins in the work of Ko\v{c}ube\u{\i} \cite{Ko75} and Gorbachuk 
and Gorbachuk \cite{Gorbachuk} with developments  from many authors, 
(see   \cite{MikhailetsS,Kuzhel1,Kuzhel2,Brasche,Storozh,KK04,Ryz07,Pos07}).
In this context, the theory of the Weyl-$M$-function was developed by Derkach and Malamud \cite{DM87,Derkach}, where spectral properties of the operator were investigated via the $M$-function and Kre\u{\i}n-type resolvent formulae were established. 
 For adjoint pairs of abstract operators, boundary triplets were introduced by Vainerman \cite{Vai80} 
and Lyantze and Storozh \cite{LS83}. Many of the results proved for the symmetric case  have subsequently been extended to this situation: see, for instance, 
Malamud and Mogilevski \cite{MMM1} for adjoint pairs of operators,  
  and Malamud and Mogilevski
\cite{MMM2,MM02} for adjoint pairs of linear relations. Amrein and Pearson \cite{AP04} generalised several results from the classical Weyl-$m$-function for
 the one-dimensional Sturm-Liouville problem to the case of Schr\"odinger operators, calling them $M$-functions,
  in particular they were able to show nesting results for families of
  $M$-functions on spherical exterior domains in $\mathbb R^3$. 
For a recent contribution with
applications to PDEs and characterisation of eigenvalues as poles of an operator valued Weyl-$M$-function, we refer the reader to \cite{BMNW08}. Further recent developments in this area can be found in \cite{BGP06,BL07,GMZ07,Post08}.
There has also been extensive work on Dirichlet-to-Neumann maps, also 
sometimes known as Poincar\'{e}-Steklov operators, especially in the inverse problems 
literature. These operators have physical meaning, associating, for instance, a surface 
current to an applied voltage and
are, in some sense, the natural PDE realisation of the abstract $M$-function 
which appears in the theory of boundary triplets discussed above.

Systems of PDEs and even ODEs occur naturally in physical applications
(reaction-diffusion equations, Max\-well systems, Dirac systems, Lam\'e systems) 
and there is much interest in  the spectral properties of operators generated by these. In Grubb \cite{G74} and Geymonat and Grubb \cite{GG77}
such systems are extensively discussed and {\it inter alia } points of essential spectrum are characterised by failure of 
ellipticity of the operator or the boundary condition. An alternative abstract approach for block operator matrices has also been developed 
(see for example Atkinson et al.\ \cite{ALMS94}).

In this paper, we shall combine results obtained 
by  Grubb on extensions of operators, see for example \cite{G68}, 
with recent results obtained by Brown, Marletta, Naboko 
and Wood \cite{BMNW08} 
on the relation between the $M$-function and the spectrum, 
to examine the spectral behaviour of  boundary value problems.
$M$-functions are defined for general closed extensions, and
associated with realisations of elliptic operators. In
particular, we shall consider both ODE and PDE examples   
where it is possible for the operator to possess spectral points that can not be detected  by the $M$-function (unlike the classical Sturm-Liouville case).

In PDE cases, the kernel of the maximal realisation has infinite
dimension; then unbounded operators between boundary spaces must be
allowed, and it is important to choose the representations of the
boundary mappings in an efficient way. We here rely on the calculus of
pseudodifferential operators ($\psi $do's), as introduced through
works of 
Calderon, Zygmund, Mihlin, Kohn, Nirenberg, H\"o{}rmander, Seeley and
others around the 1960's, as well as the calculus of pseudodifferential
boundary operators ($\psi $dbo's) introduced by Boutet de Monvel
\cite{B66, B71} and applied and extended by Grubb \cite{G74}--\cite{G96} 
and others.

\vskip.1cm
\noindent
{\bf Plan of the paper.}
Section 2 contains a discussion of the abstract theory. We begin by 
recalling the
universal pa\-ra\-me\-tri\-za\-tion of closed extensions $\wA$ established 
in
\cite{G68}, based on an invertible reference operator $A_\beta $, and
show how it applies to operators $\wA-\lambda $ (by use of techniques from
\cite{G74}), giving rise to a Kre\u{\i}n resolvent formula and
characterisations of kernels and ranges, in terms of an abstract
boundary operator $T^\lambda :V_\lambda \to W_{\bar\lambda }$. Next,
we connect this with the boundary triplets theory, as presented in
\cite{BMNW08}. We first show that for the realisation $A_B$ defined by a
boundary condition $\Gamma _1u=B\Gamma _0u$ (for a special choice of
$\Gamma _0,\Gamma _1$), the holomorphic operator family $M_B(\lambda
)$ defined for $\lambda \in\varrho (A_B)$ is homeomorphic to the
inverse of the holomorphic family $T^\lambda $ defined for $\lambda
\in \varrho (A_\beta )$, when both exist. This takes care of a special
class of boundary conditions. The idea is now developed further to include general
extensions by considering mappings between subspaces, as
in \cite{G68}. In the present context, this replaces the need to work with relations.
$M$-functions are now defined for all closed extensions $\wA$, and the
operator families $M_{\wA}(\lambda )$ and $T^\lambda $ together
describe spectral properties of the operator.

In Section 3, the ideas are implemented for realisations of elliptic
operators on smooth domains $\Omega $ in $n$-space. For second-order strongly
elliptic operators it is shown in detail how $T^\lambda $ and
$M_{\wA}(\lambda )$, for Neumann-type boundary conditions $\gamma
_1u=C\gamma _0u$, are carried over to mappings $L^\lambda $ and
$M_L(\lambda )$ between Sobolev spaces over $\partial\Omega $. The
general closed realisations give rise to operator families $L_1^\lambda $ and
$M_{L_1}(\lambda )$ between closed subspaces of  $L_2(\partial\Omega )$. For
systems and higher-order operators, the normal elliptic boundary
conditions give rise to $M$- and $L$-functions between products of
Sobolev spaces over $\partial\Omega $.

Section 4 addresses the inverse question: Are the
spectral properties fully described by  $M_{\wA}(\lambda )$ and
$T^\lambda $? The answer is in the affirmative for $\lambda \in \varrho (\wA)\cup \varrho
(A_\beta )$, and this is sufficiently informative in many situations.
But it is not so in general: We show, both by a PDE and an 
ODE matrix example, that there exist
cases where the $M$-function is holomorphic across points in the
essential spectrum of $\wA$ (and of $A_\beta $).

The authors thank the referees for careful reading of this paper and useful suggestions for improvements.

\section{Universal parametrization and $M$-functions}\label{sec2}
\subsection{A universal parametrization of closed extensions}\label{sec:2.1}

It is assumed in this paper that there is given a pair of
closed, densely defined operators $\Ami$ and  $\Ami'$ in a Hilbert
space $H$ (a so-called adjoint pair) such that the adjoint of $\Ami'$ 
is an extension of $\Ami$
and the adjoint of $\Ami$ is an extension of $\Ami'$;
we call these adjoints $\Ama$ resp.\ $\Ama'$. 
Moreover we assume that there is given a closed,
densely defined operator $A_\beta $ lying between $\Ami$
and $\Ama$ and having a bounded everywhere defined inverse. Thus we get:
\begin{align}
\Ami&\subset A_\beta \subset \Ama,\quad 0\in\varrho (A_\beta ),\quad
\Ama=(\Ami')^*,\nonumber\\
\Ami'&\subset A_\beta ^*\subset \Ama',\quad 0\in\varrho (A_\beta
^*),\quad \Ama'=\Ami^*;
\label{2.1}
\end{align}
here $\varrho (B)$ denotes the resolvent set of $B$. We call $A_\beta
$ the reference operator. Let $\cal M$ and $\cal M'$ denote the sets
of operators $\wA$ lying between $\Ami$ and
$\Ama$, resp.\ $\wA'$ lying between $\Ami'$ and $\Ama'$. We write $Au$ for $\wA u$
when $\wA\in\cal M$, resp.\ $A'v$ for $\wA 'v$ when $\wA'\in\cal M'$.
When $U$ is a closed subspace of $H$, we denote by $f_U$ the
orthogonal projection of $f$ onto $U$; the projection map is denoted $\pr_U$.

Denote  also
\begin{equation}
\operatorname{ker}\Ama = Z,\quad \operatorname{ker}\Ama '= Z';\label{2.2}
\end{equation}
and let
\begin{align}
\pr_\beta &=A_\beta ^{-1}\Ama  : D(\Ama)\to D(A_\beta ),\quad
\pr_\zeta =I-\pr_\beta :D(\Ama)\to Z,\nonumber\\
\pr_{\beta'} &=(A_\beta^*) ^{-1}\Ama'  : D(\Ama')\to D(A_\beta ^*),
\quad\pr_{\zeta '}=I-\pr_{\beta '}:D(\Ama')\to Z'.
\label{2.3}
\end{align}
Then $\pr_\beta $ and $\pr_\zeta $, resp.\ $\pr_{\beta'} $ and
$\pr_{\zeta '}$,  
are complementary  projections
defining the direct sum decompositions
\begin{equation}
D(\Ama)=D(A_\beta )\dot+ Z,\text{ resp. }D(\Ama')=D(A_\beta ^*)\dot+ Z'.\label{2.4}
\end{equation}
We also write $\pr_\beta u=u_\beta $, $\pr_\zeta  u=u_\zeta  $, etc.

The above statements are verified in \cite{G68}, which also showed the abstract
Green's formula
\begin{equation}
(A u,v) - (u, A' v)=((Au)_{Z'}, v_{\zeta '})-(u_\zeta ,
(A'v)_Z), \text{ for }u\in D(\Ama), v\in D(\Ama');\label{2.5}
\end{equation}
and we recall that in that paper,
all the closed operators in $\cal M$ were characterised  by abstract
boundary conditions:

\begin{theorem}\label{Theorem2.1} There is a one-to-one correspondence between
all closed operators $\wA\in\cal M$ and all operators $T:V\to W$,
where $V$ and $W$ are closed subspaces of $Z$ resp.\ $Z'$,
and $T$ is closed with domain $D(T)$ dense in $V$. Here $T: V\to W$
is defined from $\wA$ by 
\begin{align}
D(T)=\pr_\zeta  D(\widetilde A),&\quad V=\overline{D(T)},
\quad W=\overline{\pr_{\zeta '} D(\widetilde A^*)},\nonumber\\
Tu_\zeta&=(Au)_W \text{ for }u\in D(\wA);\label{2.6}
\end{align}
and $\wA$ is defined from $T:V\to W$ by
\begin{equation}
 D(\widetilde A)=\{ u\in D(\Ama)\mid u_\zeta \in D(T),\,
(Au)_W=Tu_\zeta \};\label{2.7}
\end{equation}
it can also be described by
\begin{equation}
u\in D(\wA)\iff u=v+z+A_\beta ^{-1}(Tz+f), v\in D(\Ami), z\in D(T),
f\in Z'\ominus W.\label{2.8}
\end{equation}

All closed subspaces $V\subset Z$ and $W\subset Z'$ and all closed densely
defined operators $T:V\to W$ are reached in this correspondence.

When $\wA$ corresponds to $T:V\to W$,
\begin{align}
\operatorname{ker}\wA&=\operatorname{ker}T,\nonumber\\
\operatorname{ran}\wA&=\operatorname{ran}T+(H\ominus W),\label{2.9}
\end{align}
orthogonal sum. In particular, $\wA$ is Fredholm if and only if $T$ is so, with the
same kernel and cokernel.

If $\wA$, hence also $T$, is injective, the inverse satisfies
\begin{equation}
\wA^{-1}=A_\beta ^{-1}+T^{-1}\pr_W, \text{ defined on
}\operatorname{ran}\wA.\label{2.10}
\end{equation}

The adjoint $\wA^*$ corresponds to $T^*:W\to V$ in the analogous way. In
particular, in the case where  $\Ami=\Ami'$ and $A_\beta $ is
selfadjoint (called the symmetric case), $\wA$ is selfadjoint if and only if $V=W$ and $T$ is
selfadjoint.
\end{theorem}

\begin{rem}\label{Remark2.2} The characterisation is related to that of Vishik
\cite{V52},
 but differs
in an important way: Vishik was concerned with {\it normally solvable}
operators $\wA$ (those with closed range), 
and his operators
between  subspaces of $Z$ and $Z'$ map in the opposite direction of those
in \cite{G68}, covering only a subset of them. In contrast, the theory in 
\cite{G68} allowed the characterisation of {\it all} closed operators
in $\cal M$.
\end{rem}

There are also some results in \cite[Section II.3]{G68} on non-closed
extensions. 

Now consider the situation where a spectral parameter $\lambda \in
{\mathbb C}$ is subtracted from the operators in $\cal M$. When $\lambda
\in\varrho (A_\beta )$, we have a similar situation as above:
\begin{equation}
\Ami-\lambda \subset A_\beta -\lambda \subset \Ama-\lambda ,\quad 
\Ami'-\bar\lambda \subset A_\beta ^*-\bar\lambda \subset \Ama'-\bar\lambda ,
\label{2.11}
\end{equation}
and we use the notation $\cal M_\lambda $, $\cal M'_{\bar\lambda
}$, and
\begin{align}
\operatorname{ker}(\Ama -\lambda )= Z_\lambda ,&\quad
\operatorname{ker}(\Ama '-\bar\lambda )= Z'_{\bar\lambda } ;\nonumber\\
\pr^\lambda _\beta =(A_\beta -\lambda )^{-1}(A -\lambda ),&\quad
\pr^{\bar\lambda }_{\beta'} =(A_\beta^*-\bar\lambda ) ^{-1}(A'-\bar\lambda ) ,\nonumber\\
\pr^\lambda _\zeta =I-\pr^\lambda _\beta ,&\quad\pr^{\bar\lambda }_{\zeta '}=I-\pr^{\bar\lambda }_{\beta'} .\label{2.12}
\end{align}
Then we have an immediate corollary of Theorem \ref{Theorem2.1}:

\begin{cor}\label{Corollary2.3} Let $\lambda \in \varrho (A_\beta )$. There is a {\rm 1--1} correspondence between
the closed operators $\wA-\lambda $ in $\cal M_\lambda $ and the closed,
densely defined operators  $T^\lambda :V_\lambda \to W_{\bar\lambda }$,
where $V_\lambda $ and $W_{\bar\lambda }$ are closed subspaces of $Z_\lambda
$ resp.\ $Z'_{\bar\lambda }$; here
\begin{align}
 D(T^\lambda )&=\pr^\lambda _\zeta  D(\widetilde A),\quad V_\lambda =\overline{D(T^\lambda )},
\quad W_{\bar\lambda }=\overline{\pr^{\bar\lambda }_{\zeta '} D(\widetilde A^*)},\nonumber\\
T^\lambda u^\lambda _\zeta&=((A-\lambda )u)_{W_{\bar\lambda }} \text{ for }u\in D(\wA),\nonumber\\
 D(\widetilde A)&=\{ u\in D(\Ama)\mid u^\lambda _\zeta \in D(T^\lambda ),\,
((A-\lambda )u)_{W_{\bar\lambda }}=T^\lambda u^\lambda _\zeta \}.\label{2.13}
\end{align}
In this correspondence,
\begin{align}
\operatorname{ker}(\wA-\lambda )&=\operatorname{ker}T^\lambda ,\nonumber\\
\operatorname{ran}(\wA-\lambda )&=\operatorname{ran}T^\lambda +(H\ominus W_{\bar\lambda } ),\label{2.14}
\end{align}
orthogonal sum. In particular, if $\lambda \in \varrho (\wA)$,
\begin{equation}
(\wA-\lambda )^{-1}=(A_\beta -\lambda )^{-1}+\ij_{V_\lambda \to H}(T^\lambda )^{-1}\pr_{W_{\bar\lambda } }.\label{2.15}
\end{equation}

\end{cor}

Here $\ij_{X\to Y}$ denotes the injection
$X\hookrightarrow Y$. 

The formula (\ref{2.15}) can be regarded as a universal ``Kre\u\i{}n resolvent
formula''. It relates the resolvents of the arbitrary operator $\wA$
and the reference operator $A_\beta $ in a straightforward way, such
that information on the spectrum of $\wA$ can be deduced from
information on $T^\lambda $. Note also the formulas in (\ref{2.14}), which
give not only a correspondence between
kernel dimensions and range codimensions, but an identification between
kernels and cokernels themselves.

The resolvent of $\wA$ was studied in \cite{G74}, from which we extract the
following additional information.  Define, for $\lambda \in \varrho (A_\beta  )$, the bounded operators on $H$:
\begin{align}
E^\lambda &=\Ama(A_\beta -\lambda )^{-1}=I+\lambda (A_\beta -\lambda )^{-1},\nonumber\\
F^\lambda &=(\Ama-\lambda )A_\beta ^{-1}=I-\lambda A_\beta ^{-1},\nonumber\\
E^{\prime\bar\lambda }&=\Ama'(A^*_\beta -\bar\lambda )^{-1}=I+\bar\lambda (A^*_\beta -\bar\lambda )^{-1}=(E^\lambda )^*,\nonumber\\
F^{\prime\bar\lambda }&=(\Ama'-\bar\lambda )(A^*_\beta )^{-1}=I-\bar\lambda (A^*_\beta )^{-1}=(F^\lambda )^*;\label{2.16}
\end{align}
then $E^\lambda $ and $F^\lambda $  are inverses of one another, and
so are $E^{\prime\bar\lambda }$ and $F^{\prime\bar\lambda }$. In
particular, the operators restrict to homeomorphisms
\begin{align}
E_Z^\lambda :Z&\simto Z_\lambda ,\quad F_Z^\lambda :Z_\lambda \simto Z,\nonumber\\
E^{\prime\bar\lambda} _{Z'}:Z'&\simto Z'_{\bar\lambda },\quad F_{Z'}^{\prime\bar\lambda }:Z'_{\bar\lambda }\simto Z'.\label{2.17}
\end{align}
Moreover, for $u\in D(\Ama)$, $v\in D(\Ama')$,
\begin{align}
\pr^\lambda _\zeta u&=E^\lambda \pr_\zeta u ,\quad \pr^\lambda _\beta u=\pr_\beta u-\lambda (A_\beta
-\lambda )^{-1}\pr_\zeta u,\nonumber\\
\pr^{\bar\lambda }_{\zeta'}v &=E^{\prime\bar\lambda }\pr_{\zeta '}v,\quad \pr^{\bar\lambda }_{\beta '}v=\pr_{\beta '}v-\bar\lambda (A^*_\beta
-\bar\lambda )^{-1}\pr_{\zeta'}v.\label{2.18}
\end{align}

This was shown in \cite[ Sect.\ 2]{G74}, in the symmetric case, and the (elementary)
proofs extend verbatim to the general case. Similar mappings occur
frequently in the literature on extensions. The following theorem
extends \cite[ Prop.\ 2.6]{G74}, to the non-symmetric situation, with
practically the same proof:

\begin{theorem}\label{Theorem2.4} For $\lambda \in\varrho (A_\beta )$, define  the operator $G^\lambda $ from $Z$ to $Z'$ by
\begin{equation}
G^\lambda z=-\lambda \pr_{Z'}E^\lambda z,\quad  z\in Z.\label{2.19}
\end{equation}
Then \begin{align}
D(T^\lambda )&=E^\lambda D(T),\quad  V_\lambda =E^\lambda V,\quad  W_{\bar\lambda }=E^{\prime \bar\lambda }W,\nonumber\\
(T^\lambda E^\lambda v,E^{\prime\bar\lambda} w)&=(Tv,w)+(G^\lambda v,w),  \text{ for }v\in
D(T),\;  w\in W.\label{2.20}
\end{align}
\end{theorem}

\begin{proof} The first line in (\ref{2.20}) follows from (\ref{2.13}) in view of (\ref{2.18}). The second line in calculated as follows:
For $u\in D(\wA)$, $w\in W$,
\begin{align}
(Tu_\zeta ,w)&=(Au,w)=(Au, F^{\prime\bar\lambda }E^{\prime\bar\lambda
}w)=(F^\lambda Au,E^{\prime\bar\lambda }w) \nonumber\\
&=((A-\lambda )(A_\beta  )^{-1} Au,E^{\prime\bar\lambda }w)= ((A-\lambda )u_\beta
,E^{\prime\bar\lambda }w) \nonumber\\
&=((A-\lambda )u ,E^{\prime\bar\lambda }w)- ((A-\lambda )u_\zeta  ,E^{\prime\bar\lambda
}w) \nonumber\\
&=(T^\lambda u^\lambda _\zeta  ,E^{\prime\bar\lambda }w)+ (\lambda u_\zeta
,E^{\prime\bar\lambda }w) =(T^\lambda E^\lambda u _\zeta  ,E^{\prime\bar\lambda }w)+ (\lambda E^\lambda u_\zeta  ,w). \nonumber
\end{align}
This shows the equation in (\ref{2.20}) when we set $u_\zeta =v$.
\end{proof}

Denote by $E^\lambda _V$ the restriction of $E^\lambda $ to a mapping from $V$ to
$V_\lambda $, with inverse $F^\lambda _V$, and let similarly $E^{\prime\bar\lambda} _W$ be the restriction of $E^{\prime\bar\lambda }$ to a
mapping from $W$  to
$W_{\bar\lambda }$, with inverse $F^{\prime\bar\lambda }_W$. 
Then the second line of (\ref{2.20}) can be written
\begin{equation}
(E_W^{\prime\bar\lambda })^*T^\lambda E_V^\lambda=T+G_{V,W}^\lambda  \text{ on }D(T)\subset V,\label{2.21}
\end{equation}
where \begin{equation}
G^\lambda _{V,W}=\pr_WG^\lambda \ij_{V\to Z}.\label{2.22}
\end{equation}
Equivalently,
\begin{equation}
T^\lambda =(F^{\prime\bar\lambda }_W)^*(T+G_{V,W}^\lambda )F^\lambda _V \text{ on }D(T^\lambda )\subset V_\lambda .\label{2.23}
\end{equation} Then the Kre\u\i{}n resolvent formula (\ref{2.15}) can be made more explicit
as follows:

\begin{cor}\label{Corollary2.5}
When $\lambda \in \varrho (\wA)\cap \varrho (A_\beta )$, $T^\lambda $ is invertible, and 
\begin{equation}
(T^\lambda )^{-1}=E^\lambda _V(T+G_{V,W}^\lambda
)^{-1}(E^{\prime\bar\lambda }_W)^*.\label{2.24}
\end{equation}
Hence
\begin{equation}
(\wA-\lambda )^{-1}=(A_\beta -\lambda )^{-1}+\ij_{V_\lambda \to H}E^\lambda _V(T+G_{V,W}^\lambda
)^{-1}(E^{\prime\bar\lambda }_W)^*
\pr_{W_{\bar\lambda }  }.\label{2.25}
\end{equation}
\end{cor}

\begin{proof} (\ref{2.24}) follows from (\ref{2.23}) by inversion, and insertion in
(\ref{2.15}) shows (\ref{2.25}). 
\end{proof}

Note that $G^\lambda _{V,W}$ depends in a simple way on $V$ and $W$ and is
independent of $T$.

\subsection{  Connections between the universal parametrization and
boundary triplets} \label{sec:2.2}

The setting for boundary triplets in the non-symmetric case is the
following, according to  \cite{BMNW08} (with reference to \cite{LS83}, 
\cite{MM02}): $\Ami$, $\Ama$, $\Ami'$
and $\Ama'$ are given as in the beginning of Section \ref{sec:2.1}, and 
there is given a pair of Hilbert spaces $\cal H$, $\cal K$ and two
pairs of
``boundary operators''
\begin{equation}
\begin{pmatrix}\Gamma_1\\ \quad\\\Gamma_0 \end{pmatrix}: D(\Ama)\to \begin{matrix}{\cal H}\\\times\\{\cal K}\end{matrix}, \quad
\begin{pmatrix}\Gamma'_1\\ \quad\\ \Gamma'_0 \end{pmatrix}: D(\Ama')\to
\begin{matrix}{\cal K}\\ \times\\{\cal H}\end{matrix},
\label{2.26}
\end{equation}
bounded with respect to the graph norm and  surjective, such that
\begin{equation}
(A u,v) - (u, A' v)=(\Gamma _1u, \Gamma '_0v)_{\cal H}-(\Gamma _0u ,
\Gamma '_1v)_{\cal K}, \text{ all }u\in D(\Ama), v\in D(\Ama'),\label{2.27}
\end{equation}
and 
$$
D(\Ami)=D(\Ama)\cap \operatorname{ker}\Gamma _1\cap
\operatorname{ker}\Gamma _0,\quad 
D(\Ami')=D(\Ama')\cap \operatorname{ker}\Gamma '_1\cap
\operatorname{ker}\Gamma '_0. 
$$

Note that under the assumption of (\ref{2.1}),
the choice
\begin{equation}
{\cal H}=Z', \;{\cal K}=Z,\; \Gamma _1=\pr_{Z'} \Ama,\;\Gamma _0=\pr_\zeta ,\;
\Gamma '_1=\pr_Z\Ama',\;\Gamma '_0=\pr_{\zeta '} ,\label{2.28}
\end{equation}
defines in view of (\ref{2.4}) and (\ref{2.5}) a boundary triplet.

Following \cite{BMNW08}, the boundary triplet is used to define 
operators $A_B\in \cal M$ and $A'_{B'}\in\cal M'$ for any pair of
operators $B\in \cal L(\cal K,\cal H)$, $B'\in \cal L(\cal H,\cal K)$
by
\begin{equation}
D(A_B)=\operatorname{ker}(\Gamma _1-B\Gamma _0),\quad
D(A'_{B'})=\operatorname{ker}(\Gamma '_1-B'\Gamma '_0). \label{2.29}
\end{equation}

In order to discuss resolvents,
\cite{BMNW08} assumes that $\varrho (A_B)\ne \emptyset$, which means that the
situation with existence of an invertible $A_\beta $ as in (\ref{2.1}) can
be 
obtained at
least after subtraction of a spectral parameter $\lambda _0$. Thus we can
 build on the full assumption (\ref{2.1}) from now on. (Theorem \ref{Theorem2.1} shows
that there is an abundance of different invertible operators in $\cal M$ then.)

\begin{defn}\label{Definition2.6} For $\lambda \in \varrho (A_B)$, the
$M$-function $M_B(\lambda )$ is defined by
$$
M_B(\lambda ): \operatorname{ran}(\Gamma _1-B\Gamma _0)\to {\cal
K},\quad M_B(\Gamma _1-B\Gamma _0)u=\Gamma _0u\text{ for all }u\in Z_\lambda ; 
$$
and for $\lambda \in \varrho (A'_{B'})$, the
$M$-function $M'_{B'}(\lambda )$ is defined similarly by
$$
M'_{B'}(\lambda ): \operatorname{ran}(\Gamma '_1-B'\Gamma '_0)\to {\cal
H},\quad M'_{B'}(\Gamma '_1-B'\Gamma '_0)v=\Gamma '_0v\text{ for all }v\in Z'_\lambda . 
$$
\end{defn}
 
It is shown in \cite{BMNW08} that when $\varrho (A_B)\ne \emptyset$,
\begin{equation}
(A_B)^*=A'_{B^*}.\label{2.30}
\end{equation}

We shall set all this in relation to the universal parametrization,
when the boundary triplet is chosen as in (\ref{2.28}). {\it We assume
}(\ref{2.28}) {\it from now on.}

Concerning $A_B$, note that $B$ is taken as a bounded operator from $Z$ to
$Z'$, and that
$$
D(A_B)=\{u\in D(\Ama)\mid (Au)_{Z'}=Bu_\zeta \},
$$
by definition. This shows that for the operator $T:V\to W$ that $A_B$ corresponds
to by Theorem \ref{Theorem2.1},
$$
V=Z,\quad W=Z',\quad T=B.
$$
Note that (\ref{2.30}) follows from Theorem \ref{Theorem2.1}.

When $Z$ and $Z'$ are finite dimensional, all operators $B$ will be
bounded. But in the case where $\dim Z=\dim Z'=\infty $, Theorem \ref{Theorem2.1}
shows that unbounded $T$'s must be allowed, to cover general
extensions. Therefore we in the following
take $B$ closed, densely defined and possibly unbounded, and
define $A_B$ by
\begin{equation}
D(A_B)=\{u\in D(\Ama)\mid  u_\zeta \in D(B), (Au)_{Z'}=Bu_\zeta \}.
\label{2.30a}
\end{equation}

\begin{lemma}\label{Lemma2.7} $\operatorname{ran}(\Gamma _1-B\Gamma _0)=Z'$. In
fact, any $f\in Z'$ can be
written as $f=(\Gamma _1-B\Gamma _0)v=\Gamma _1v$ for $v\in D(A_\beta
)$ taken equal to $A_\beta ^{-1}f$.
\end{lemma}

\begin{proof} Let $v$ run through $D(A_\beta
)$. Then $\Gamma _0v=\pr_\zeta v=0\in D(B)$, and $Av$ runs through
$H=(\operatorname{ran}\Ami)\oplus Z'$, so $\Gamma _1v-B\Gamma
_0v=(Av)_{Z'}$ runs through $Z'$.  In other words, we can take
$v=A_\beta ^{-1 }f$, for any $f\in Z'$.
\end{proof}

\begin{lemma}\label{Lemma2.8} For any $\lambda \in \varrho (A_B)$, $M_B(\lambda
)$ is well-defined as a mapping from $Z'$ to $Z$ by 
$$
 M_B(\lambda)(\Gamma _1-B\Gamma _0)u=\Gamma _0u\text{ for all }u\in Z_\lambda
\text{ with }\Gamma _0u\in D(B), 
$$
also when $D(B)$ is a subset of $Z$.
In fact,
\begin{equation}
M_B(\lambda )=\pr_\zeta (I-(A_B-\lambda )^{-1}(\Ama-\lambda ))A_\beta
^{-1}\ij_{Z'\to H}.\label{2.30c}
\end{equation}

\end{lemma}

\begin{proof} In the defining equation, we can now only allow  those
$u=z^\lambda \in Z_\lambda $ for which $\pr_\zeta z^\lambda \in D(B)$. 

We first show: When $f=(\Gamma _1-B\Gamma _0)v$ for 
  $v=A_\beta^{-1}f$ as in Lemma \ref{Lemma2.7}, then there is a $z^\lambda
\in Z_\lambda $ with $\Gamma _0z^\lambda \in D( B)$ such that  
\begin{equation}
f=(\Gamma _1-B\Gamma _0)z^\lambda.\label{2.30b}
\end{equation}
For, let 
$x=(A_B-\lambda )^{-1}(A-\lambda )v$, then $v-x\in Z_\lambda
$. Moreover, since $x\in D(A_B)$, $\pr_\zeta x\in D(B)$ in view of
(\ref{2.30a}), and $\pr_\zeta v=0\in D(B)$, as already noted. Hence 
$$
(\Gamma _1-B\Gamma _0)(v-x)=(\Gamma _1-B\Gamma _0)v=f;
$$
so we can take $z^\lambda =v-x$. We conclude:
$$
\{(\Gamma _1-B\Gamma _0)z^\lambda \mid z^\lambda \in
Z_\lambda ,\; \pr_\zeta z^\lambda \in D(B) \}=Z'.
$$

Next, we note that the $z^\lambda $ in (\ref{2.30b}) is uniquely determined
from $f$. For, if $f=0$ and $z^\lambda $ solves (\ref{2.30b}), then
$z^\lambda \in
D(A_B)=D(A_B-\lambda )$, which is linearly independent from $Z_\lambda
$ since $\lambda \in \varrho (A_B)$, so $z^\lambda =0$.

Thus we can for any $f\in Z'$ set $M_B(\lambda )f= \pr_\zeta z^\lambda $
where $z^\lambda $ is the unique solution of (\ref{2.30b}); this defines
a linear mapping $M_B(\lambda )$ from
$Z'$ to $Z$. The procedures used above to construct $M_B(\lambda )$ are summed up
in (\ref{2.30c}).

\end{proof}

Since $B$ is closed, $M_B(\lambda )$ is closed, hence continuous, as a
mapping from $Z'$ to $Z$.

For a further analysis of $M_B(\lambda )$, assume
$\lambda \in \varrho (A_\beta )$. Then the maps $E^\lambda $,
$F^\lambda $ etc.\ in (\ref{2.16}) are defined. 
Let $z^\lambda \in Z_\lambda $, and consider the defining equation
\begin{equation}
M_B(\lambda )((Az^\lambda )_{Z'}-B\pr_\zeta z^\lambda )=\pr_\zeta z^\lambda ;\label{2.31}
\end{equation}
where
$\pr_\zeta z^\lambda $ is required to lie in $D(B) $. 
 By (\ref{2.17}) and (\ref{2.18}), there is a unique $z\in Z$ such that 
$z^\lambda =E^\lambda _{Z}z$; in fact
$$
z^\lambda =E^\lambda _{Z}z=\pr^\lambda _\zeta z; \quad z=F_Z^\lambda z^\lambda =\pr_\zeta z^\lambda ,
$$so the requirement is that $z^\lambda \in E^\lambda _Z D(B)$.

Writing (\ref{2.31}) in terms of $z$, and using that $Az^\lambda =\lambda
z^\lambda $, we find:
\begin{align}
z&=\pr_\zeta z^\lambda =M_B(\lambda )((A z^\lambda )_{Z'}-Bz )= 
M_B(\lambda )((\lambda  z^\lambda )_{Z'}-Bz )\nonumber\\
&= 
M_B(\lambda )((\lambda  E^\lambda z )_{Z'}-Bz )= M_B(\lambda )(-G^\lambda  -B)z,
\label{ 2.32a}
\end{align}
cf.\ (\ref{2.19}). Here we observe that the operator to the right of
$M_B(\lambda )$ equals $T^\lambda $ from Section \ref{sec:2.1} up to
homeomorphisms:
\begin{equation}
-G^\lambda -B=-(G^\lambda +T)=-(E_{Z'}^{\prime\bar\lambda })^*
T^\lambda E_Z^\lambda ,\label{2.32}
\end{equation}
by (\ref{2.21}); here $T^\lambda $ is invertible from $E^\lambda _ZD(B)$ onto
$Z'_\lambda $. We conclude that $M_B(\lambda )$ is the inverse of the
operator in (\ref{2.32}). We
have shown:

\begin{theorem}\label{Theorem2.9} When the boundary triplet is chosen as in
{\rm (\ref{2.28})} and $\lambda \in \varrho (A_B)\cap \varrho (A_\beta )$, $-M_B(\lambda )$ 
equals the inverse of $B+G^\lambda =T+G^\lambda $,
also equal to the inverse of $T^\lambda $
modulo homeomorphisms:
\begin{equation}
-M_B(\lambda )^{-1}=B+G^\lambda =T+G^\lambda =(E_{Z'}^{\prime\bar\lambda })^*T^\lambda E_Z^\lambda.\label{2.33}
\end{equation}
In particular, $M_B(\lambda )$ has range $D(B)$.
\end{theorem}


With this insight we have access to the straightforward resolvent
formula 
(\ref{2.25}), which implies in this case:

\begin{cor}\label{Corollary2.10} For $\lambda \in \varrho (A_B)\cap \varrho (A_\beta )$,
\begin{equation}
(A_B-\lambda )^{-1}=(A_\beta -\lambda )^{-1}-\ij_{Z_\lambda \to H}E^\lambda _Z
M_B(\lambda ) (E^{\prime \bar\lambda }_{Z'})^*
\pr_{Z'_{\bar\lambda } }.\label{2.34}
\end{equation}
\end{cor} 

\begin{remark}
Other Kre\u{\i}n-type resolvent formulae for realisations in the general framework of relations can be found in \cite[Section 5.2]{MM02}.
\end{remark}

We also have the
direct link between null-spaces and ranges (\ref{2.14}), when merely $\lambda
\in \varrho (A_\beta )$.

\begin{cor}\label{Corollary2.11} For any $\lambda \in \varrho (A_\beta )$,
\begin{align}
\operatorname{ker}(A_B-\lambda )&=E^\lambda _Z
\operatorname{ker}(B+G^\lambda ), \nonumber\\
\operatorname{ran}(A_B-\lambda )&=(F^{\prime\bar\lambda }_{Z'})^*\operatorname{ran}(B+G^\lambda )+\operatorname{ran}(\Ami -\lambda )
.\label{2.35}
\end{align}

\end{cor}

For $\lambda \in \varrho (A_\beta )$, this adds valuable information to the
results from \cite{BMNW08} on the connection between eigenvalues of
$A_B$ and poles of $M_B(\lambda )$.

The analysis moreover implies that $M_B(\lambda
)$ and $M'_{B^*}(\bar\lambda )$ are adjoints, at least when $\lambda
\in \varrho (A_\beta )$.

Observe that $T+G^\lambda $ (and $T^\lambda $) is well-defined for all $\lambda \in \varrho
(A_\beta )$, whereas  $M_B(\lambda )$ is well-defined for all $\lambda
\in \varrho (A_B)$; the latter fact is useful for other 
purposes. In this
way, the two operator families complement each other, and,
together, contain much spectral
information. 

It is noteworthy that the widely studied boundary
triplets theory leads to an operator family whose elements are inverses
of elements of the operator family generated by Theorem \ref{Theorem2.1} -- compare with
Remark \ref{Remark2.2} on the connection with Vishik's theory.

\subsection{  The $M$-function for arbitrary closed extensions} \label{sec:2.3}

The above considerations do not fully use the potential of
Corollary \ref{Corollary2.3}, Theorem \ref{Theorem2.4} and Corollary \ref{Corollary2.5}, which allow much more
general boundary operators $T:V\to W $. But, inspired by the result
in Theorem \ref{Theorem2.9}, we can in fact establish useful $M$-functions in all
these other cases, namely homeomorphic to the inverses of the
operators
$T^\lambda $ that exist for $\lambda \in \varrho (\wA)\cap \varrho
(A_\beta )$, and extended to exist for all $\lambda \in \varrho (\wA)$.

\begin{theorem}\label{Theorem 2.12} Let $\wA$ be an arbitrary closed densely
defined operator between $\Ami$ and $\Ama$, and let $T:V\to W$ be the
corresponding operator according to Theorem
{\rm \ref{Theorem2.1}}. For any $\lambda \in \varrho (\wA)$ there is a bounded operator
$M_{\wA}(\lambda ):W\to V$, depending holomorphically on
$\lambda \in \varrho (\wA)$, such that when
$\lambda \in \varrho (A_\beta )$,  $-M_{\wA}(\lambda )$ is the
inverse of $T+G^\lambda _{V,W}$, and is homeomorphic to $T^\lambda $
(as defined in
Section {\rm \ref{sec:2.1})}.
It satisfies
\begin{equation}
M_{\wA}(\lambda )((Az^\lambda )_W-T\pr_\zeta z^\lambda )=\pr_\zeta z^\lambda ,
\label{2.36}
\end{equation}
for all $z^\lambda \in Z_\lambda $ such that $\pr_\zeta z^\lambda \in
D(T)$. Its definition extends to all $\lambda \in \varrho (\wA)$ by
the formula
\begin{equation}
M_{\wA}(\lambda )=\pr_\zeta \bigl(I-(\wA-\lambda )^{-1}
(\Ama-\lambda )\bigr)A_\beta ^{-1}\ij_{W\to H} \label{2.41}
\end{equation}

In particular, the resolvent formula
\begin{equation}
(\wA-\lambda )^{-1}=(A_\beta -\lambda )^{-1}-\ij_{V_\lambda \to H}E^\lambda _V
M_{\wA}(\lambda ) (E^{\prime \bar\lambda }_{W})^*
\pr_{W_{\bar\lambda } }\label{2.37}
\end{equation}
holds when $\lambda \in \varrho (\wA)\cap \varrho (A_\beta )$. For all
$\lambda \in \varrho (A_\beta )$,
\begin{align}
\operatorname{ker}(\wA-\lambda )&=E^\lambda _V
\operatorname{ker}(T+G^\lambda _{V,W}),\nonumber\\
\operatorname{ran}(\wA-\lambda )&=(F^{\prime\bar\lambda }_{W})^*
\operatorname{ran}(T+G^\lambda _{V,W})+H\ominus W_{\bar\lambda }.\label{2.38}
\end{align}
\end{theorem}

\begin{proof} 
Following the lines of proofs of Lemma \ref{Lemma2.7} and \ref{Lemma2.8}, we define 
$M_{\wA}(\lambda )$ satisfying (\ref{2.36}) as
follows: Let $f\in W$. Let $v=A_\beta ^{-1}f$; then $\pr_\zeta v=0\in
D(T)$, and \begin{equation}
(Av )_W-T\pr_\zeta v =Av=f.\label{2.39}
\end{equation}
Next, let $x =(\wA-\lambda
)^{-1}(A-\lambda )v$, then $z^\lambda =v-x$ lies in $Z_\lambda $ and
satisfies $\pr_\zeta z^\lambda \in D(T)$ (since $\pr_\zeta v=0$ and
$\pr_\zeta x\in D(T)$, cf.\ (\ref{2.6})). This $z^\lambda $ satisfies
\begin{equation}
(Az^\lambda )_W- T\pr_\zeta z^\lambda =f,\label{2.40}
\end{equation}
in view of (\ref{2.39}) and the fact that $x\in D(\wA)$.

Next, observe that for any vector $z^\lambda \in Z_\lambda $ with
$\pr_\zeta z^\lambda \in D(T)$ such
that (\ref{2.40}) holds, $f=0$ implies 
$z^\lambda =0$, since such a $z^\lambda $ lies in the two linearly
independent spaces $D(\wA-\lambda )$ and $Z_\lambda $. So there is
indeed a mapping from $f$ to $\pr_\zeta z^\lambda $ solving (\ref{2.40}),
for any $f\in W$, defining $M_{\wA}(\lambda )$. It is described by
(\ref{2.41}).
The holomorphicity in $\lambda \in \varrho (\wA)$ is seen from this formula.

The mapping is connected with $T^\lambda $ (cf.\ Corollary \ref{Corollary2.3}) as
follows: 

When $\lambda \in \varrho (\wA)\cap \varrho (A_\beta )$, then
$z=\pr_\zeta z^\lambda =F^\lambda _Vz^\lambda $, and $z^\lambda
=E^\lambda _Vz$, so the vectors $z^\lambda $ with $\pr_\zeta z^\lambda
\in D(T)$ constitute the space $E^\lambda _VD(T)$. Calculating as in
(\ref{ 2.32a}) we then find that
\begin{align*}
z&=\pr_\zeta z^\lambda =M_{\wA}(\lambda )((A z^\lambda )_{W}-Tz )= 
M_{\wA}(\lambda )((\lambda  z^\lambda )_{W}-Tz )\nonumber\\
&= 
M_{\wA}(\lambda )((\lambda  E^\lambda z )_{W}-Bz )= M_{\wA}(\lambda )(-G^\lambda _{V,W} -T)z,
\end{align*}
so $M_{\wA}(\lambda )$ is the inverse of $-(T+G^\lambda
_{V,W}):D(T)\to W$. The remaining statements follow from Corollary \ref{Corollary2.3}
and Corollary \ref{Corollary2.5}.
\end{proof}
Note that when $M_{\wA}(\lambda )$ is considered in a neighbourhood of
a spectral point of $\wA$ in $\varrho (A_\beta )$, then we have not
only information on the possibility of a pole of $M_{\wA}(\lambda )$,
but an inverse $T^\lambda $, from which
$\operatorname{ker}(\wA-\lambda )$ and $\operatorname{ran}(\wA-\lambda
)$ can be read off.

\section{Applications to elliptic partial differential operators}\label{sec3}
\subsection{Preliminaries}\label{sec3.1}  

For elliptic operators $A$ defined over an open subset $\Omega $ 
of ${\mathbb R}^n$,
$n>1$, the null-space of the
maximal realisation is infinite dimensional, so that there is much
more freedom of choice of boundary spaces and mappings than in
ODE cases. 
It is necessary to allow unbounded operators between boundary spaces
to obtain a theory covering the well-known cases. Moreover, there is the
problem of regularity of domains: For a given realisation $\wA$
representing a boundary condition, it is not always certain that
$\wA^*$ represents an analogous boundary condition, but this can often
be assured if $D(\wA)$ is known to be contained in the most regular Sobolev
space $H^m(\Omega )$, where $m$ is the order of $A$; this holds when
the boundary condition is {\it elliptic}.

The theory of pseudodifferential boundary operators (Boutet
de Monvel \cite{B66}, \cite{B71}, and  e.g.\
Grubb \cite{G90}--\cite{G08}) is
known as an efficient tool in the treatment of boundary value problems on
smooth sets (we call it the $\psi $dbo calculus for short, similarly
to the customary use of $\psi $do for
pseudodifferential operator). 
A guiding principle in the construction of
general theories would therefore be to make it possible to use the
$\psi $dbo calculus in applications to concrete operators. The
$\psi $dbo calculus is a theory for genuine {\it operators}
and their approximate solution operators, with many structural
refinements; it has not been customary to study {\it relations} in this context. We therefore find it adequate to
interpret  the realisations of elliptic operators in terms of the
theory based on \cite{G68}, that characterises the elements in $\cal
M$ by
operators, 
rather 
than  relations.

Let $\Omega $ be a smooth subset of ${\mathbb R}^n$, with $C^\infty $
boundary $\partial\Omega =\Sigma $, let $m$ be a positive  integer, 
and let $A=
\sum_{|\alpha |\le m}a_\alpha (x)D^\alpha $ be an $m$-th order differential operator on
$\Omega $ with coefficients in $C^\infty (\comega)$ and uniformly
elliptic
(i.e., the principal symbol $a^0(x,\xi )=\sum_{|\alpha |=m}a_\alpha (x)\xi ^\alpha $ is invertible for all $x\in
\comega$, all $\xi \in{\mathbb R}^n\setminus \{0\}$). The maximal and
minimal realisations in $H=L_2(\Omega )$ act like $A$ in the distribution sense, with
domains defined by
\begin{equation}
D(\Ama)=\{u\in L_2(\Omega )\mid Au\in L_2(\Omega )\}, \quad D(\Ami)=H^m_0(\Omega ),\label{3.1}
\end{equation}
 and it is well-known (from ellipticity arguments) that 
\begin{equation}
\Ami^*=\Ama',\quad \Ama^*=\Ami',\label{3.2}
\end{equation}
where $\Ama'$ and $\Ami'$ are the analogous operators for the formal
adjoint $A'$ of $A$. 

The operators belonging to $\cal M$ resp.\ $\cal M'$ defined as in Section \ref{sec:2.1}
are called the {\it realisations} of $A$ resp.\
$A'$. 

We denote by $H^s(\Omega )$ the 
Sobolev space over $\Omega $ of order $s$, namely 
the space of restrictions to $\Omega $ of the elements of $H^s({\mathbb
R}^n)$, which consists of the distributions $u\in \cal S'$ such that
$(1+|\xi |^2)^{s/2}\hat u\in L_2({\mathbb R}^n)$. (${\cal S}'={\cal
S}'({\mathbb R}^n)$ is Schwartz' space of temperate distributions.) By $H^s_0(\comega )$ we
denote the subspace of $H^s({\mathbb R}^n)$ of elements supported in
$\comega$. 
For $s>-\frac12$, $s-\frac12$ not integer, this space can be
identified with  the
closure of $C_0^\infty (\Omega )$ in $H^s(\Omega )$, also denoted $H^s_0(\Omega )$. 
Sobolev spaces over $\Sigma $, $H^s(\Sigma )$, are defined  by use of
local coordinates.
We denote $\gamma
_ju=(\partial_n  ^ju)|_{\Sigma }$, where $\partial_n $ is the
derivative along the interior normal $\vec n $ at $\Sigma $. Here $\gamma
_j$ maps $H^s(\Omega )\to H^{s-j-\frac12}(\Sigma  )$ for
$j<s-\frac12$. For $j<m$ there is an extension $\gamma _j:D(\Ama)\to
H^{-j-\frac12}(\Sigma  )$, cf.\ e.g.\ Lions and Magenes \cite{LM68}.

Let us briefly recall the relevant elements of the $\psi $dbo
calculus. In its general form it treats operators 
\begin{equation}\label{3.2a}
\mathcal A=\begin{pmatrix} P_++G&K\\ &\\ T&S\end{pmatrix}:
\begin{matrix} C^\infty(\overline\Omega )^N&&C^\infty(\overline\Omega)^{N'}\\
\times&\to&\times\\ C^\infty(\Sigma )^M&&C^\infty(\Sigma )^{M'}\end{matrix}.
\end{equation}
Here $T$ is a generalized {\it trace operator}, going from $\Omega$ to $\Sigma $; $K$ is a so-called {\it Poisson operator}
(called a potential operator or coboundary
operator in some other
texts), going from $\Sigma $ to $\Omega$; $S$ is a {\it
pseudodifferential operator on} $\Sigma $; and $G$ is an
operator on $\Omega$ called a {\it singular Green operator},
 a
non-pseudodifferential term that has to be included in order to have
adequate composition rules. $P$ is a $\psi $do defined on an open set
$\widetilde \Omega \supset\comega$, and $P_+$ is its truncation to
$\Omega $,
defined by
$P_+ u=r^+Pe^+u$, where $r^+$ restricts $ {\cal
D}'(\widetilde{\Omega })$ to ${\cal D}'({\Omega })$, and $e^+$ extends locally
integrable functions on $\comega$ by zero on $\widetilde\Omega
\setminus \comega $. $P$ is assumed to satisfy the so-called
transmission condition at $\Sigma $; this holds for the operators
derived from elliptic differential operators that we consider
here. There are suitable Sobolev space mapping properties in terms of
the orders of the entering operators.

A solvable elliptic PDE problem 
\begin{equation}\label{3.2b}
Au=f\text{ on }\Omega ,\quad  Tu=\varphi \text{ on }
\Sigma ,
\end{equation} 
enters in this framework by an operator (where we suppress
the index + on $A$ since it acts locally) 
\begin{equation}\label{3.2c}
\begin{pmatrix} A\\ \quad\\T\end{pmatrix}:C^\infty (\comega )^N\to\begin{matrix} C^\infty(\overline\Omega )^N\\
\times\\ C^\infty(\Sigma )^{M'}\end{matrix} 
\end{equation}
(note that $M=0$ and $M'>0$), with the inverse
\begin{equation}\label{3.2d}
\begin{pmatrix} R&K\end{pmatrix}:  \begin{matrix} C^\infty(\overline\Omega )^N\\
\times\\ C^\infty(\Sigma )^{M'}\end{matrix}\to C^\infty (\comega )^N;\quad R=Q_++G,
\end{equation}
where $Q_++G$ solves the problem \eqref{3.2b} with $\varphi =0$ and
$K$ solves the problem \eqref{3.2b}
with $f=0$.
Here $Q$ is a parametrix of $A$ on $\widetilde \Omega $ (for example
in case $A=-\Delta $, $Q$ is the convolution with $c_n|x|^{2-n}$ on ${\mathbb R}^n$ when
$n\ge 3$), and $G$ is the correction term needed to make $Q_++G$ map
into the functions satisfying the homogeneous boundary condition.

Besides providing a convenient terminology, the $\psi $dbo calculus
has the advantage that it gives complete composition rules:
When $\cal A$ and $\cal A'$ are two systems as in \eqref{3.2a}, the
composed operator $\cal A\cal A'$ again has this structure. In
particular, a composition $TK$ gives a $\psi $do on $\Sigma $, and a
composition $KT$ gives a singular Green operator. Compositions $TP_+$,
$TG$ and $ST$ give trace operators, compositions $P_+K$, $GK$ and $KS$
give Poisson operators. These are the facts that we shall mainly use in the
present paper. Details on the $\psi $dbo calculus are found e.g.\ in
\cite{G96}, \cite{G08}.

\subsection{A typical second-order case}
\label{sec3.2}

To give an impression of the theory, we begin by studying in some detail the case of a
second-order strongly elliptic operator $A$. This part is divided into five subsections. In the first one we introduce boundary triplets for the operator $A$. In the next three subsections we concentrate on the case of ``pure conditions'', i.e.~when $T:Z\to Z'$. For this case, we show in \ref{sec:trans} how $T$ can be identified with an operator $L$ representing a Neumann-type boundary condition. In \ref{sec:mfn}, we study the corresponding $M$-function, proving, among other things, a Kre\u\i{}n-type resolvent formula. Subsection \ref{sec:ell} takes a closer look at problems with elliptic boundary conditions. Finally, in \ref{sec:sub}, we consider the general case when $T:V\to W$ and $V,W$ are subspaces of $Z$ and $Z'$, respectively.

\subsubsection{Boundary triplets}
We begin by introducing boundary triplets for the case of a
second-order strongly elliptic operator $A$
 i.e., with
$\operatorname{Re}a^0(x,\xi )\ge c_0|\xi |^2$ for $x\in\comega $ and 
$\xi \in{\mathbb
R}^n$ ($c_0>0$), taking $\Omega $ bounded. 
Let $s_0(x)$ be the (nonvanishing) coefficient of $-\partial_n ^2$ when $A$ is
written in normal and tangential coordinates at a boundary point $x$,
then
$A$ has the  Green's formula 
\begin{equation}
(Au,v)_{L_2(\Omega )}-(u,A'v)_{L_2(\Omega )}=(s_0\gamma _1u,\gamma _0v)_{L_2(\Sigma )}-
(\gamma _0u, \bar s_0\gamma _1v+{\cal A}'_0\gamma _0v)_{L_2(\Sigma )},\label{3.3}
\end{equation}
 for $u,v\in H^2(\Omega )$, with a suitable first-order differential
operator ${\cal A}'_0$ over $\Sigma $. We denote $s_0\gamma
_1=\nu_1 $, $\bar s_0\gamma _1=\nu_1 '$.

 A simple example was
explained in \cite[ Section 7]{BMNW08}, namely
\begin{equation}
A=-\Delta +p(x)\cdot \grad ,\text{ with formal adjoint }A'v=-\Delta v
-\divg(\bar p v);\label{3.4}
\end{equation}
where $p$ is an $n$-vector of functions in $C^\infty (\comega )$. 

We can
assume, after addition of a constant to $A$ if necessary, that the
Dirichlet problem for $A$ is uniquely solvable.

The Dirichlet realisation $A_\gamma $ is the operator lying in $\cal M$ with
domain
$$
D(A_\gamma )=D(\Ama)\cap H^1_0(\Omega)=H^2(\Omega )\cap H^1_0(\Omega)
$$
(the last equality follows by elliptic regularity theory); it has $0\in\varrho (A_\gamma )$.
Let 
\begin{equation}
Z^s_\lambda (A)=\{u\in H^s(\Omega )\mid (A-\lambda )u=0\},\label{3.5}
\end{equation}
for $s\in{\mathbb R}$. It is known from \cite{LM68} that the trace
operators $\gamma _0$ and $\gamma _1$, hence also $\nu_1 $, extend by
continuity to continuous maps
\begin{equation}
\gamma _0: Z^s_\lambda (A)\to H^{s-\frac12}(\Sigma ),\quad
\gamma _1\text{ and }\nu_1 : Z^s_\lambda (A)\to H^{s-\frac32}(\Sigma ),\label{3.6}
\end{equation}
for all $s\in{\mathbb R}$. When $\lambda \in \varrho (A_\gamma )$, let
$K^\lambda _{\gamma }:\varphi \mapsto u$ denote the Poisson operator solving the
semi-homogeneous Dirichlet problem
\begin{equation}
(A-\lambda )u=0\text{ in }\Omega ,\quad \gamma _0u= \varphi .\label{3.7}
\end{equation}
 It maps continuously
\begin{equation}H^{s-\frac12}(\Sigma )\to
H^{s}(\Omega ), \text{ for all }s\in{\mathbb R}.
\end{equation}
Moreover, it maps
$H^{s-\frac12}(\Sigma )$ homeomorphically onto 
 $ Z^s_\lambda (A)$
for all $s\in {\mathbb R}$, with $\gamma _0$ acting as an inverse
there.
(We introduce below a special notation for the restricted operator when
$s=0$, see \eqref{3.17}.)
Analogously, there is a Poisson operator $K^{\prime \bar\lambda
}_\gamma $ solving (\ref{3.7}) with $A-\lambda $ replaced by $A'-\bar\lambda
$, mapping $H^{s-\frac12}(\Sigma )$ homeomorphically onto
$Z^s_{\bar\lambda }(A')$,  with $\gamma _0$ acting as an inverse there.

Now define the Dirichlet-to-Neumann operators for each $\lambda \in
\varrho (A_\gamma )$, 
\begin{equation}
P^\lambda _{\gamma _0,\nu_1 }=\nu_1 K^\lambda _\gamma ;\quad 
P^{\prime\bar\lambda }_{\gamma _0,\nu_1 '}=\nu_1 'K^{\prime\bar\lambda }_\gamma ;
 \label{3.9}
 \end{equation}
they are a first-order elliptic pseudodifferential operators over
$\Sigma $, continuous and Fredholm from
$H^{s-\frac12}(\Sigma )$ to $H^{s-\frac32}(\Sigma )$
for all $s\in{\mathbb R}$ (details e.g.\ in \cite{G71}).

We shall use the notation for general trace maps $\beta $ and $\eta $:
\begin{equation}
P^\lambda _{\beta ,\eta }: \beta u\mapsto \eta u,\quad u\in
Z^s_\lambda (A), 
  \label{3.10}
 \end{equation}
when this operator is well-defined.

Introduce the trace operators $\Gamma $ and $\Gamma '$ (from \cite{G68},
where they were called $M$ and $M'$)  by 
\begin{equation}
\Gamma u=\nu_1 u-P^0_{\gamma _0,\nu_1 }\gamma _0u,\quad \Gamma'
u=\nu_1 'u-P^{\prime 0}_{\gamma _0,\nu_1 '}\gamma _0u. 
 \label{3.11}
 \end{equation}
Here $\Gamma $ maps $D(\Ama)$ continuously onto
$H^{\frac12}(\Sigma )$ and can also be written $\Gamma
=\nu_1 A_\gamma ^{-1}\Ama$, and $\Gamma '$ has the analogous properties.
Moreover,  a generalised Green's formula is valid {\it for all} 
$u\in D(\Ama)$, $v\in D(\Ama')$:
\begin{equation}
(Au,v)_{L_2(\Omega )}-(u,A'v)_{L_2(\Omega )}=(\Gamma  u,\gamma _0v)_{\frac12,
-\frac12}-(\gamma _0u, \Gamma 'v)_{-\frac12,\frac12},\label{3.12} 
\end{equation}
where $(\cdot,\cdot)_{s ,-s }$ denotes the duality pairing
between $H^s (\Sigma )$ and $H^{-s }(\Sigma
)$. Furthermore,
\begin{equation}
(Au,w)_{L_2(\Omega )}=(\Gamma u,\gamma _0w)_{\frac12,-\frac12}\text{ for all
}w\in Z^0_0(A').\label{3.13}
\end{equation}
(Cf.\cite [ Th. III 1.2]{G68}.)

To achieve $L_2(\Sigma )$-dualities in the right-hand side of
(\ref{3.12}), one can choose the norms in $H^{\pm\frac12}(\Sigma )$ to be
induced by suitable isometries from the norm in $L_2(\Sigma )$. There
exists a family of pseudodifferential elliptic invertible
operators $\Lambda _s$ of order $s\in {\mathbb R}$ on $\Sigma $,
 symmetric with respect to the duality in $L_2(\Sigma )$ and with $\Lambda _{-s}=\Lambda _s^{-1}$, such that
when each $H^s(\Sigma )$ is provided with the norm for which $\Lambda
_s$ is an
isometry from  $H^s(\Sigma )$ onto $L_2(\Sigma )$, $\Lambda _s$ also maps
 $H^{t}(\Sigma
)$ isometrically onto $H^{t-s}(\Sigma )$, all $t$, and 
 \begin{equation}
(\Lambda _{-s}\varphi , \Lambda _s \psi
)_{s,-s}=(\varphi ,\psi )_{L_2(\Sigma )}, \quad \varphi
,\psi \in L_2(\Sigma ).\label{3.14}
\end{equation}
Then when we introduce composed operators
\begin{equation}
\Gamma _1=\Lambda _{\frac12} \Gamma , \quad \Gamma '_1=\Lambda _{\frac12} \Gamma ', \quad 
\Gamma _0=\Lambda _{-\frac12}\gamma _0=\Gamma '_0;\label{3.15}
\end{equation}
(\ref{3.12}) takes the form (\ref{2.27}) with ${\cal H=\cal
K}=L_2(\Sigma )$:

\begin{prop}\label{Proposition3.1} For the adjoint pair $\Ami$ and $\Ami'$, {\rm (\ref{3.15})} provides a boundary
triplet  with $
{\cal H=\cal K}=L_2(\Sigma )$:
\begin{equation}
(Au,v)_{L_2(\Omega )}-(u,A'v)_{L_2(\Omega )}=(\Gamma _1 u,\Gamma '_0v)_{L_2(\Sigma
)}-(\Gamma _0u, \Gamma _1'v)_{L_2(\Sigma )},\label{3.16}
\end{equation}
holds when $u\in D(\Ama)$, $v\in D(\Ama')$:
\end{prop}

Such reductions to $L_2$-dualities are made in \cite{BMNW08} and \cite{Pos07}.
\cite{G68} did not make the modification by composition with $\Lambda _{\pm\frac12}$, but
worked directly with (\ref{3.12}). (This was in order to avoid introducing
too many operators. Another reason was that the Sobolev spaces
$H^s (\Sigma )$ do not have a ``preferred norm'' when
$s \ne 0$; only the duality $(\cdot,\cdot)_{s ,-s }$
should be consistent with the self-duality of $L_2(\Sigma
)$. Moreover, when the realisation $\wA$ represents an {\it elliptic}
boundary condition, $D(\wA)\subset H^2(\Omega )$ and the boundary
values are in $L_2(\Sigma )$. ---
Various homeomorphisms were used in
\cite{G74} for the sake of numerical comparison.)

In the rest of this section, {\it we use the abbreviation $H^s$ for $H^s(\Sigma )$}. 
We shall keep the formulation with dualities in the study of
pure Neumann-type boundary conditions, but return to (\ref{3.15}) in
connection with more general boundary conditions.

\subsubsection{Interpretation of the boundary conditions}\label{sec:trans}
Consider the set-up of Section 2.1 with $A_\beta =A_\gamma $, the
projection $\pr_\beta $ being denoted $\pr_\gamma $. The realisation
$A_\gamma $ itself of course corresponds to the case
$V=W=\{0\}$ in Theorem \ref{Theorem2.1}. 

Let $\wA$ be a closed realisation which corresponds to an operator $T$ with
$V=Z$, $W=
Z'$ by Theorem \ref{Theorem2.1}. Note that 
\[Z=Z^0_0(A), \;Z'=Z^0_0(A')\text{, and for }
\lambda \in \varrho (A_\gamma ), \;Z_\lambda =Z^0_\lambda (A),\;
Z'_{\bar\lambda }=Z^0_{\bar\lambda }(A');\]
 closed subspaces of $L_2(\Omega )$.

Denote the
restrictions of $\gamma _0$ to mappings from $Z_\lambda $ resp.\ $Z'_{\bar\lambda }$ to
$H^{-\frac12}$ by $\gamma _{Z_\lambda }$ resp.\
$\gamma _{Z'_{\bar\lambda }}$; they are homeomorphisms
\begin{equation}
\gamma _{Z_\lambda }: Z_\lambda \simto H^{-\frac12},\quad 
\gamma _{Z'_{\bar\lambda }}: Z'_{\bar\lambda }\simto H^{-\frac12},
\label{3.17}
\end{equation}
and their inverses $\gamma _{Z_\lambda }^{-1}$ resp.\ $\gamma
_{Z'_{\bar\lambda }}^{-1}$ coincide with $K^{\lambda }_\gamma$ resp.\
$K^{\prime\bar\lambda }_\gamma$ but have the restricted range
space. Their adjoints map
\[ 
\gamma _{Z_{\lambda }}^* : H^{\frac12}\simto Z_{\lambda
},\quad
\gamma _{Z'_{\bar\lambda }}^* : H^{\frac12}\simto Z'_{\bar\lambda }.
\]
When $\lambda =0$, the $\lambda $-indications are left out.

We shall first interpret $\wA$ in terms of a boundary
condition using the maps with $\lambda =0$; this stems from
\cite{G68}. 
The above homeomorphisms allow ``translating'' an operator $T:Z\to Z'$ to
an operator $L:H^{-\frac12}\to H^\frac12$, as in the diagram
\begin{equation}\label{tag3.20b}
\CD
Z     @>  \gamma _Z  >>    H^{-\frac12}\\
@VTVV           @VV  L  V\\
   Z'  @>>(\gamma _{Z'}^*)^{-1} >   H^{\frac12}
\endCD 
\hskip1cm D(L)=\gamma _ZD(T),
 \end{equation}
where the horizontal maps are homeomorphisms. In other words,
\begin{equation}
 L=(\gamma _{Z'}^*)^{-1}T\gamma _Z^{-1},\text{ with }D( L)
=\gamma _ZD(T)
\label{3.20}
\end{equation} 
 (a closed densely defined operator from  $H^{-\frac12}$ to
$H^{\frac12}$). Hereby we have, when $\varphi =\gamma _Z z\in D(L)$,
$\psi =\gamma _{Z'}w\in H^{-\frac12}$,
\[
(Tz,w)_{Z'}=(T\gamma
_Z^{-1}\varphi ,\gamma _{Z'}^{-1}\psi )_{Z'}= (L\varphi ,\psi )_{\frac12,-\frac12}. 
\]
Note that $D(L)=\gamma _0D(T)=\gamma _0\pr_\zeta D(\wA)=\gamma _0D(\wA)$,
since $\gamma _0$ vanishes on $D(A_\gamma )$.

Recall the
equation defining $T$ from $\wA$:
\begin{equation}\label{3.17a}
(T\pr_\zeta u,w)=(Au,w)\text{ for }u
\in D(\wA),\; w\in Z'.
\end{equation}
In view of (\ref{3.13}), the right-hand side may be written
 \begin{equation}
(Au,w)=(\Gamma u,\gamma _0w)_{\frac12,-\frac12}=(\nu_1 u-P_{\gamma _0,\nu_1 }\gamma _0u,\gamma _0w)_{\frac12,-\frac12}
\text{ for
all }w\in Z'.\label{3.18}
 \end{equation}
For the left-hand side we have with $L$ defined above, using that $\gamma _Z\pr_\zeta u=\gamma _0(u-\pr_\gamma u)=\gamma _0u$,
\[
(T\pr_\zeta u,w)=(L\gamma _Z\pr_\zeta u,\gamma _{Z'}
w)_{\frac12,-\frac12}=(L\gamma _0u,\gamma _0w)_{\frac12,-\frac12}. 
\]
Then, when we write $\gamma _0w=\psi $, \eqref{3.17a} takes the form
\begin{equation}
(L\gamma _0u ,\psi )_{\frac12,-\frac12}=(\nu_1 u
-P^0_{\gamma _0,\nu_1}\gamma _0u,\psi )_{\frac12,-\frac12}
\text{ for all }u\in
D(\wA), \text{ all }\psi \in H^{-\frac12}.   
\end{equation}
Since $\psi $ runs through $H^{-\frac12}$, this may be written $L\gamma
_0u=\nu_1 u-P^0_{\gamma _0,\nu_1}\gamma _0u$, or, 
 \begin{equation}
 \nu_1 u=(L+P^0_{\gamma _0,\nu_1})\gamma _0u, \quad \gamma
_0u\in D(L).\label{3.22}
\end{equation}
So in fact $\wA$ represents a Neumann-type boundary condition \eqref{3.22}.

Conversely, if we want $\wA$ to represent a given Neumann-type boundary condition
\begin{equation}
\nu_1 u=C\gamma _0u,\label{3.23}
\end{equation}
where $C$ is a $\psi $do  
over $\Sigma $, we see that $L$ has
to be taken to act like
\begin{equation}
L=C-P^0_{\gamma _0,\nu_1}.\label{3.24}
\end{equation}

Now let us turn to the $\lambda $-dependent case. Here we consider the
families $\wA-\lambda $
and $T^\lambda $ and can proceed in a very similar way. When working
with the concrete boundary Sobolev spaces we find the advantage that $Z$ and
$Z_\lambda $ are mapped by $\gamma _0$ to the same space
$H^{-\frac12}$. In fact, 
\begin{equation}\label{3.24a}
\gamma
_{Z_\lambda }=\gamma _Z F^\lambda _Z,\quad \gamma _Z=\gamma _{Z_\lambda
}E^\lambda 
_Z,\text{ and similarly }\gamma
_{Z'_{\bar\lambda }}=\gamma _{Z'} F^{\prime\bar\lambda }_{Z'},\quad
\gamma
_{Z'}=\gamma _{Z'_{\bar\lambda }} E^{\prime\bar\lambda }_{Z'},
\end{equation} 
since, e.g., $\gamma _0F^\lambda u=\gamma _0(u-\lambda A_\gamma
^{-1}u)=\gamma _0u$, cf.\ (\ref{2.16}).

Let $\lambda \in \varrho (A_\gamma )$. In the defining equation for
$T^\lambda $,
 \begin{equation}
(T^\lambda \pr^\lambda _\zeta u,w)=((A-\lambda )u,w)\text{ for }u
\in D(\wA), w\in Z'_{\bar\lambda },\label{3.25}
\end{equation}
we rewrite the two sides as
\begin{align*}
 (T^\lambda u^\lambda _\zeta ,w)&=(L^\lambda \gamma _{Z_\lambda
}u^\lambda _\zeta , \gamma _{Z'_{\bar\lambda }}w)_{\frac12,-\frac12},\\
((A-\lambda )u,w)&=(\nu_1 u-P^\lambda _{\gamma _0,\nu_1}\gamma _0u,\gamma _0
w )_{\frac12,-\frac12}, 
\end{align*}
where $L^\lambda
:H^{-\frac12}\to H^\frac12$ is defined by
\begin{equation}
L^\lambda =(\gamma _{Z'_{\bar\lambda }}^*)^{-1}T^\lambda \gamma
_{Z_\lambda }^{-1},\quad D(L^\lambda )=\gamma _{Z_\lambda }D(T^\lambda );
\label{3.26}
 \end{equation}
note that 
\begin{equation}
D(L^\lambda )=\gamma _0D(T^\lambda )=\gamma _0E^\lambda D(T )
=\gamma _0 D(T)=D(L).
\end{equation}
Then since $ \gamma _{Z_\lambda
}u^\lambda _\zeta =\gamma _0u$, the operator $\wA-\lambda $ represents the
boundary condition
\begin{equation}
L^\lambda\gamma _0u=\nu_1 u-P^\lambda _{\gamma _0,\nu_1}\gamma
_0u,\quad \gamma _0u\in D(L^\lambda )=D(L).\label{3.27}
\end{equation}

Moreover, in view of Corollary \ref{Corollary2.5}, $L^\lambda $ is related to $T+G^\lambda $ as follows:
\[
\CD
Z     @>{E^\lambda _{Z}} >>Z _\lambda     @>  \gamma _{Z_\lambda }  >>    H^{-\frac12}\\
@V  T+G^\lambda  VV @VT^\lambda VV           @VV  L^\lambda   V\\
 Z' @>>(F^{\prime\bar\lambda }_{Z'})^* >  Z'_{\bar\lambda }
@>>(\gamma _{Z'_{\bar\lambda }}^*)^{-1} >   H^{\frac12}\endCD
\]
where the horizontal maps are homeomorphisms. In view of
\eqref{3.24a},  they compose to $\gamma _Z$ resp.\
$(\gamma _{Z'}^*)^{-1}$, so  
\begin{equation}
L^\lambda 
 =(\gamma ^*_{Z' })^{-1}(T+G^\lambda )\gamma _{Z }^{-1}
=L+(\gamma ^*_{Z' })^{-1}G^\lambda \gamma _{Z }^{-1}.
\label{3.27a}
\end{equation}

Since $D(\wA-\lambda )=D(\wA)$, (\ref{3.27}) and (\ref{3.22}) define the same
boundary condition, hence
\begin{equation}
L^\lambda =L+ P^0 _{\gamma _0,\nu_1}-P^\lambda _{\gamma _0,\nu_1 }\text{ on }D(L).\label{3.28}
\end{equation}

\begin{remark}\label{Remark3.2a} In particular, it can be inferred (e.g.\ from the case $L=0$) that $(\gamma ^*_{Z' })^{-1}G^\lambda \gamma
_{Z }^{-1}=P^0 _{\gamma _0,\nu_1}-P^\lambda _{\gamma _0,\nu_1 }$. 
Note how the operator family $L^\lambda $ (in this case where $V=Z$, $W=Z'$) is
written as the sum of a $\lambda $-independent operator 
$L$ (defining the domain of the realisation) and a $\lambda
$-dependent operator
$P^0 _{\gamma _0,\nu_1}-P^\lambda _{\gamma _0,\nu_1} $, which
is universal in the sense that it only depends on $A$, the set $\Omega
$, 
and $\lambda $.
It is useful to observe
that since $G^\lambda $ is continuous from $Z$ to $Z'$ for each
$\lambda $, $P^0 _{\gamma _0,\nu_1}-P^\lambda _{\gamma _0,\nu_1 }$ is continuous from $H^{-\frac12}$ to $H^\frac12$, hence is of
order $-1$, in contrast to its two individual terms that are
elliptic of order 1 (having the same principal symbol). 
\end{remark}

This analysis results in the theorem:

\begin{theorem}\label{Theorem3.4} For the second-order strongly elliptic operator
$A$ introduced above, let $\wA$ be a closed realisation with $\pr_\zeta
D(\wA)$ dense in $Z$ and $\pr_{\zeta '}D(\wA^*)$ dense in $Z'$. Let
$T:Z\to Z'$ be the operator it corresponds to by Theorem {\rm \ref{Theorem2.1}}. 

{\rm (i)} When $Z$ and
$Z'$ are mapped to $H^{-\frac12}$ by $\gamma _0$ and Theorem
{\rm 2.1} is carried over to the setting based on the Green's formula
{\rm (\ref{3.12})}, $\wA$ corresponds to a
closed, densely defined operator $L:H^{-\frac12}\to
H^\frac12$ with domain $D(L)=\gamma _0D(\wA)$ such that $\wA$
represents the boundary condition
\begin{equation}
\nu_1 u=C\gamma _0u,\text{ where }C=L+P^0_{\gamma _0,\nu_1 }.\label{3.39}
\end{equation}
Here $L$ is defined from $T$ by {\rm \eqref{3.20}}.

{\rm (ii)} For any $\lambda \in \varrho (A_\gamma )$, $\wA-\lambda $ corresponds
similarly to
\begin{equation}
L^\lambda=L+P^0_{\gamma _0,\nu_1 }-P^\lambda _{\gamma _0,\nu_1 }:H^{-\frac12}\to
H^\frac12,\label{3.40}
 \end{equation}
with domain $D(L^\lambda )=D(L)$. $\wA-\lambda $ is the realisation of
$A-\lambda $ determined by the boundary condition 
\eqref{3.39}, where $C$ may also be written $C=L^\lambda +P^\lambda
_{\gamma _0,\nu_1 }$. Here, when $\wA-\lambda $ corresponds to
$T^\lambda $ by Corollary {\rm \ref{Corollary2.3}}, $L^\lambda $ is
defined from $T^\lambda $ by {\rm \eqref{3.26}}. Moreover, {\rm
  \eqref{3.27a}} holds.

{\rm (iii)} Furthermore, for any $\lambda \in \varrho (A_\gamma )$,
 \begin{align}
\operatorname{ker}(\wA-\lambda )&=K^\lambda _\gamma
\operatorname{ker}L^\lambda,\nonumber\\
\operatorname{ran}(\wA-\lambda )&=\gamma _{Z'_{\bar\lambda }}^*
\operatorname{ran}L^\lambda+\operatorname{ran}(\Ami-\lambda ).
  \label{3.41}
\end{align}
\end{theorem}

\begin{proof} All has been accounted for above except point (iii), but
this follows immediately from \eqref{2.14}. 
\end{proof}

\subsubsection{The $M$-function}\label{sec:mfn}
We  define 
an $M$-function in this 
representation, by use of Lemma 2.8. We have from (\ref{2.30c}) for $\lambda \in\varrho (\wA)$:
\begin{equation}
M_{\wA}(\lambda )=\pr_\zeta \bigl((I-(\wA-\lambda )^{-1}(\Ama-\lambda
)\bigr)A_\gamma  ^{-1}\ij_{Z'\to H}: Z'\to Z.\label{3.29}
\end{equation}
We know from Theorem \ref{Theorem2.9} that the $M$-function should coincide with
minus the
inverse of the operator induced from $T^\lambda $ in (\ref{2.32}), when $\lambda \in
\varrho (\wA)\cap \varrho (A_\gamma )$. So, applying the trace maps in
(\ref{3.17}) in a similar way as we did for $T$, we get
\begin{align}
M_L(\lambda )
&=\gamma _Z M_{\wA}(\lambda )\gamma _{Z'}^*=\gamma _Z \pr_\zeta \bigl((I-(\wA-\lambda )^{-1}(\Ama-\lambda
)\bigr)A_\gamma  ^{-1}\ij_{Z'\to H}\gamma _{Z'}^*\nonumber\\
& =\gamma _0\bigl(I-(\wA-\lambda )^{-1}(\Ama-\lambda
)\bigr)A_\gamma  ^{-1}\ij_{Z'\to H}\gamma _{Z'}^*.\label{3.30}
\end{align}
Here, when $\lambda \in \varrho (\wA)\cap\varrho (A_\gamma
)$, 
$-M_L(\lambda )$ is the inverse of the operator translated from
$T+G^\lambda $, namely, in view of (\ref{3.27a})--(\ref{3.28}),
\begin{equation}
M_L(\lambda )=-(L+ P^0 _{\gamma _0,\nu_1}-P^\lambda _{\gamma _0,\nu_1})^{-1}=-(L^\lambda )^{-1},
\label{3.31}
 \end{equation}
bounded from $H^\frac12$ to $ H^{-\frac12}$,
and then it has  range $D(L)$. Moreover, it has the 
extension
to $\lambda \in \varrho (\wA)$ given in (\ref{3.30}), a
holomorphic family of bounded operators from $H^\frac12$ to $H^{-\frac12}$.

Note that when 
$\lambda \in\varrho (\wA)\cap\varrho
(A_\gamma )$, 
$L^\lambda $ is
surjective onto $H^\frac12$. 
These considerations lead to:

\begin{theorem}\label{Theorem3.4a} For the realisation considered in
  Theorem {\rm \ref{Theorem3.4}}, there is an
 $M$-function defined by 
\begin{equation}
M_L(\lambda ) =\gamma _0\bigl(I-(\wA-\lambda )^{-1}(\Ama-\lambda
)\bigr)A_\gamma  ^{-1}\ij_{Z'\to H}\gamma _{Z'}^*,\label{3.42}
\end{equation}
a holomorphic family of bounded operators from $H^{\frac12}$ to $H^{-\frac12}$.
For $\lambda \in \varrho (\wA)\cap \varrho (A_\gamma )$, it satisfies
\begin{equation}
M_L(\lambda )=-(L+ P^0 _{\gamma _0,\gamma _1}-P^\lambda _{\gamma
_0,\gamma _1})^{-1}=-(L^\lambda )^{-1}.
\label{3.43}
\end{equation}
There is the following Kre\u\i{}n resolvent formula, valid for all $\lambda \in \varrho (\wA)\cap\varrho (A_\gamma )$:
\begin{align}
(\wA-\lambda )^{-1}&=(A_\gamma -\lambda )^{-1}-\ij_{Z_\lambda \to H}\gamma _{Z_\lambda }^{-1} 
M_L(\lambda )(\gamma _{Z'_{\bar\lambda }}^*)^{-1}\pr_{Z'_{\bar\lambda
  }}\label{3.44}\\
&=(A_\gamma -\lambda )^{-1}-K^\lambda _\gamma  
M_L(\lambda )(K^{\prime\bar\lambda}_\gamma )^*.\nonumber
\end{align}
\end{theorem}

\begin{proof} The definition of $M_L$ is accounted for above. The
first line in the  Kre\u\i{}n formula follows 
from 
(\ref{2.37}) in the case $V=Z$, $W=Z'$, by the calculation
$$
E^\lambda _Z
M_{\wA}(\lambda ) (E^{\prime \bar\lambda }_{Z'})^*=E^\lambda _Z
\gamma _Z^{-1}M_{L}(\lambda )
(\gamma _{Z'}^*)^{-1} (E^{\prime \bar\lambda }_{Z'})^*=\gamma _{Z_\lambda }^{-1}M_L(\lambda )(\gamma _{Z'_{\bar\lambda }}^*)^{-1},
$$
using (\ref{3.30}) and \eqref{3.24a}.
The second line follows since
$\ij_{Z_\lambda \to H}\gamma _{Z_\lambda }^{-1} =K^\lambda _\gamma
:H^{-\frac12}\to H$, $(\gamma _{Z'_{\bar\lambda }}^*)^{-1}\pr_{Z'_{\bar\lambda
  }}=(\ij_{Z'_{\bar\lambda }\to H}\gamma _{Z'_{\bar\lambda
  }}^{-1})^*=(K^{\prime\bar\lambda}_\gamma )^*:H\to H^{\frac12}$
(recall that $H=L_2(\Omega )$).
\end{proof}

Note that with the notation \eqref{3.10}, $L^\lambda =-P^\lambda
_{\gamma _0 ,\nu _1-C\gamma _0 }$ and  $M_L(\lambda )=P^\lambda _{\nu
  _1-C\gamma _0 ,\gamma _0}$, cf.\ \eqref{3.24}.

\subsubsection{Elliptic boundary conditions}\label{sec:ell}
Further information can be obtained in elliptic cases.
Using  the Sobolev space mapping properties of $\gamma _Z$ and
its inverse, one finds since $D(A_\gamma )\subset H^2(\Omega )$ that $D(\wA)\subset H^2(\Omega )$ if and only if
$D(L)\subset H^{\frac32}$. If $L$ is given as an arbitrary $\psi $do
of order 1, it is not in general bounded from $H^{-\frac12}$ to
$H^\frac12$, but has a subset as domain.

{\it Ellipticity} of the boundary value problem defined from
\eqref{3.23}, \eqref{3.24} (the
Shapiro-Lopatinski\u\i{}
condition)  holds precisely when $L$ 
acts like an {\it elliptic} $\psi $do of order 1. Then 
$\operatorname{ran}L\subset H^\frac12$ implies 
$D(L)\subset H^\frac32$, and the graph-norm on $D(L)$ is
equivalent with the $H^\frac32$-norm. Let $\varrho (\wA)\cap\varrho
(A_\gamma )\ne\emptyset$ (a reasonable hypothesis in our discussion)
and let $\lambda _0\in\varrho (\wA)\cap\varrho
(A_\gamma  )$. Then $L^{\lambda _0}$ has kernel and cokernel $\{0\}$
(since $T^{\lambda _0}$ has so, by Theorem 2.1) and is likewise
elliptic of order 1, so the inverse $-M_L(\lambda _0)$ is an elliptic
$\psi $do of order $-1$; it is defined on
all of $H^{\frac12}$. It maps $H^{\frac12}$ onto $H^\frac32$ (since the range
of the operator from any $H^s$ to $H^{s+1}$ has codimension 0). It
follows that
\begin{equation}
D(L)=D(L^{\lambda _0})= H^\frac32.
\label{3.32}
 \end{equation}
Then moreover, $D(\wA)\subset H^2(\Omega )$
(and $D(\wA)$ is the
largest subset of $D(\Ama)$ where (\ref{3.23}) holds).

In this case, the adjoint of $(L^{\lambda _0})^{-1}$ (as a bounded
operator from $H^{\frac12}$ to $H^{-\frac12}$) is $((L^{\lambda
  _0})^*)^{-1}$, so the adjoint of $L$ (as an unbounded closed densely defined
operator from $H^{-\frac12}$ to $H^{\frac12}$) is $L^*$ with domain
$H^\frac32$. Then $\wA^*$ is the realisation of $A'$ defined
by the elliptic boundary condition 
\begin{equation}
\nu_1'u=(L^*+P^{\prime 0}_{\gamma _0,\nu_1 '})\gamma _0u,\label{3.33}
 \end{equation}
and $D(\wA^*)\subset H^2(\Omega )$. Note that if $A=A'$ and $L$ is elliptic of order 1 and symmetric, $\wA$
is selfadjoint. 

Before including these observations in a theorem we shall show that
also $\pr_{Z_\lambda }$,
$\pr_{Z'_{\bar\lambda }}$ and their adjoints $\ij_{Z_\lambda \to H}$,
$\ij_{Z'_{\bar\lambda }\to H}$, belong to the $\psi $dbo calculus;
this will allow a discussion of $M_L(\lambda )$ for $\lambda \in
\varrho (\wA)\setminus \varrho (A_\gamma )$.

\begin{prop}\label{Proposition3.3} Let $\lambda \in\varrho (A_\gamma )$. Then
$\pr_{Z'_{\bar\lambda }}$ acts as the singular Green operator
\begin{equation}
\pr_{Z'_{\bar\lambda }}=I-(\Ama -\lambda )R((A'-\bar\lambda
)(A-\lambda ),\gamma _0,\gamma_1 )(\Ama'-\bar\lambda ),\label{3.34}
 \end{equation}
where $R((A'-\bar\lambda )(A-\lambda ),\gamma _0,\gamma_1 ):g\mapsto u$ is the solution
operator for the problem
\begin{equation}
(A'-\bar\lambda )(A-\lambda )u=g,\quad \gamma _0u=\gamma_1 u=0.\label{3.35}
 \end{equation}

Similarly, 
\begin{equation}
\pr_{Z_{\lambda }}=I-(\Ama '-\bar\lambda )R((A-\lambda
)(A'-\bar\lambda ),\gamma _0,\gamma_1 )(\Ama-\lambda ).\label{3.36}
 \end{equation}

Moreover, $\ij_{Z_\lambda \to H}$ acts like the adjoint of the operator in {\rm (\ref{3.36})}
in the $\psi $dbo calculus, and 
$\ij_{Z'_{\bar\lambda }\to H}$ acts like the adjoint of the operator
in {\rm \eqref{3.34}}.
\end{prop}

\begin{proof} When $f\in H$, $\pr_{Z'_{\bar\lambda }}f$ is the second component of
$f$ in the decomposition 
$$
H=\operatorname{ran}(\Ami-\lambda )\oplus Z'_{\bar\lambda }.
$$
Write $f=f_1+f_2$ according to this decomposition, and note that
\begin{equation}
(A'-\bar\lambda )f=(A'-\bar\lambda )f_1. \label{3.37}
\end{equation}
Observe moreover, that $v=(A_\gamma
-\lambda )^{-1}f_1$ is the unique element in $D(\Ami)=H^2_0(\Omega )$
such that $(A-\lambda )v=f_1$.  In view of (\ref{3.37}), $v$ moreover solves
\begin{equation}
(A'-\bar\lambda )(A-\lambda )v=(A'-\bar\lambda )f,\quad \gamma _0v=\gamma _1v=0,\label{3.38}
\end{equation}
and this solution is unique since $(A'-\bar\lambda )(A-\lambda )$ is
formally selfadjoint strongly elliptic with positive minimal
realisation, hence has a positive Friedrichs extension, representing its
Dirichlet problem (\ref{3.35}). Thus $v$ is uniquely determined as
$$
v=R((A'-\bar\lambda )(A-\lambda ),\gamma _0,\gamma_1 )(A'-\bar\lambda )f,
$$
and  $f_2$ is given by the formula (\ref{3.34}). It should be noted that the
operator in (\ref{3.34}) has a good meaning on $L_2(\Omega )$ in the $\psi
$dbo calculus;
first $A'-\lambda $ maps $L_2(\Omega )$ continuously into $H^{-2}(\Omega )$, then
$R((A'-\bar\lambda )(A-\lambda ),\gamma _0,\gamma _1)$ maps
$H^{-2}(\Omega )$ homeomorphically onto $H^2_0(\Omega )$, as is known
for Dirichlet problems for positive operators, and finally $A-\lambda
$ maps $H^2_0(\Omega )$ continuously into $L_2(\Omega )$. \eqref{3.34}
defines  a
singular Green operator since the $\psi $do part of the second
term cancels out with $I$.

The formula (\ref{3.36}) follows by interchanging the roles of $A-\lambda $
and $A'-\bar\lambda $.

Finally, as a technical point taken care of in \cite{G90}, the operator in (\ref{3.34}), although the factor to the right is
of order 2, is of class 0 since $R((A'-\bar\lambda )(A-\lambda
),\gamma _0,\gamma_1 )$ is of class $-2$. Then it
does have an adjoint in the $\psi $dbo calculus, so the assertion
follows since the adjoint of $\pr_{Z_\lambda }:H\to Z_\lambda $ is $\ij_{Z_\lambda \to H}$.
\end{proof}

\begin{theorem}\label{Theorem3.4b} 
For the operators considered in Theorems {\rm \ref{Theorem3.4}} and
{\rm \ref{Theorem3.4a}}, one has:

{\rm (i)} When $L$ acts like an elliptic $\psi $do of order $1$, then $D(L)\subset
H^\frac32$ and  $D(\wA)\subset
H^2(\Omega )$,
and  when 
$\lambda \in \varrho(\wA)\cap \varrho (A_\gamma )$, 
all the operators entering in the formulas belong to 
the $\psi $dbo calculus, and $M_L(\lambda )$ is elliptic of order $-1$.
 Moreover, if  $\varrho
(\wA)\cap \varrho (A_\gamma )\ne\emptyset$,
$D(L)=
H^\frac32$.

{\rm (ii)} When the statements in {\rm (i)} hold, we have moreover that
$\wA^*$ is the realisation of $A'$ determined by the elliptic boundary
condition {\rm \eqref{3.33}}, and $D(L^*)=H^\frac32$, $D(\wA^*)\subset
H^2(\Omega )$. In particular, if $A=A'$ and $L$ is symmetric and elliptic of order $1$, $\wA$
is selfadjoint with $D(\wA)\subset H^2(\Omega )$.

{\rm (iii)} 
When $L$ is elliptic of order $1$ with $D(L)=
H^\frac32$, $M_L(\lambda )$ is
an elliptic $\psi $do of order $-1$ for all $\lambda \in \varrho (\wA)$.

\end{theorem}

\begin{proof} The assertions in (i) and (ii) were shown above.
For (iii), we use the formula
(\ref{3.42}) derived in (\ref{3.30}) from (\ref{2.30c}). When
$L$ is elliptic with $D(L)=H^\frac32$, the resolvent $(\wA-\lambda
)^{-1}$ belongs to the $\psi $dbo calculus and
 maps  $L_2(\Omega )$ into $H^2(\Omega )$; then the composition 
rules in the $\psi $dbo calculus imply
that $M_L(\lambda )$ is a $\psi $do over $\Sigma $. 
Since $\bigl(I-(\wA-\lambda )^{-1}(\Ama-\lambda
)\bigr)A_\gamma  ^{-1}$ maps $L_2(\Omega )$ continuously into $H^2(\Omega )$,
$M_L(\lambda )$ maps $H^\frac12$ continuously  
 into $H^{\frac32}$, hence is of order $-1$, for all
$\lambda \in \varrho (\wA)$. Since the {\it principal}
symbol of $M_{\wA}(\lambda )$ (in the $\psi $dbo calculus) is
independent  of $\lambda $, also the principal symbol of the
$\psi $do $M_L(\lambda )$ is independent of $\lambda $; it equals the
inverse of the principal symbol of $L$ and is therefore elliptic. 
\end{proof}

\begin{remark}\label{Remark 3.4}
The Kre\u\i{}n resolvent formula  \eqref{3.44} can be compared with the
standard resolvent formula \linebreak$(\wA-\lambda )^{-1}=Q_{\lambda
  ,+}+G_\lambda $ (cf.\ Seeley \cite{Se66} and e.g.\ Grubb \cite{G96}),
where $Q_\lambda $ is a parametrix (an approximate inverse) of
$A-\lambda $ on a neighborhood $\widetilde\Omega $ of $\comega$,
$Q_{\lambda ,+}$ is its truncation to $\Omega$, and $G_\lambda $ is a
singular Green operator adapted to the boundary condition (as in \eqref{3.2d}). This has
been used e.g.\ to show asymptotic kernel and trace expansions (also for the
associated heat operator). Formula \eqref{3.44} gives more direct
eigenvalue information.
\end{remark} 

\begin{remark}\label{Remark 3.5} 
The operators $L$ and $M_L(\lambda )$ can be pseudodifferential also in
non-elliptic cases. A striking case is where $L=0$ as an operator from
$H^{-\frac12}$ to $H^\frac12$, so that $\wA$
represents the boundary condition
\begin{equation}
\nu_1 u=P^{0}_{\gamma _0,\nu_1 }\gamma _0u.\label{3.45}
\end{equation}
This is  Kre\u\i{}n's ``soft extension'' \cite{K47}, namely the realisation $A_M$ with
domain
\begin{align}
D(A_M)=D(\Ami)\dot+Z,\label{3.46}
\end{align}
as described in  \cite{G68}. Its domain is not contained in any $H^s(\Omega
)$ with $s>0$. Here $M_L(\lambda )=-(P^{0}_{\gamma _0,\nu_1 }
-P^{\lambda }_{\gamma _0,\nu
_1})^{-1}$ when this inverse exists; cf.\ (\ref{3.43}). Spectral properties
of $A_M$ are worked out in \cite{G83}.
\end{remark}

\subsubsection{The general case}\label{sec:sub}
The previous subsections cover the cases where $T$ goes from $Z$ to $Z'$ in
Theorem 2.1
(called ``pure conditions'' in \cite{G68}). The family $\cal M$ moreover
contains the operators corresponding to closed, densely defined
operators $T:V\to W$ with arbitrary closed subspaces $V\subset Z$,
$W\subset Z'$. In \cite{G68}, it was shown that these correspond to operators $L:X\to
Y^*$, $X=\gamma _0V$ and $Y=\gamma _0W$ (subspaces of $H^{-\frac12}$);  then $\wA$ represents the
boundary condition
\begin{equation}
\gamma _0u\in X,\quad(L\gamma
_0u,\psi )_{Y^*,Y}=(\Gamma u,\psi )_{\frac12,-\frac12},
\text{
all }\psi \in Y.\label{3.47}
\end{equation}
The equation is also written  $L\gamma _0u=\Gamma u|_{Y}$
(restriction as a functional on $Y$).

Let us show how this looks when we use the modified
trace operators in (\ref{3.15}) mapping the maximal domains to $L_2(\Sigma )$. Setting
\begin{equation}
X_1=\Gamma _0V=\Lambda _{-\frac12}X,\;  Y_1=\Gamma _0W=\Lambda _{-\frac12}Y, \quad
 L_1=(\Gamma _{0,W}^*)^{-1}T\Gamma _{0,V}^{-1},\text{ with }D( L_1)
=\Gamma _0D(T),\label{3.48}
\end{equation}
where $\Gamma _{0,V}$ resp.\ $\Gamma _{0,W}$ denote the restrictions
of $\Gamma _0$ as mappings
from $V$ to $ X_1$ resp.\ from $W$ to $Y_1$, we have the following diagram:
\begin{equation}\label{diagram:subspace}
\CD
V     @>  \Gamma _{0,V}  >>    X_1\\
@V T VV           @VV  L_1  V\\
   W  @>>(\Gamma _{0,W}^*)^{-1} >   Y_1
\endCD 
 \end{equation}
We find as in
(\ref{3.18})--(\ref{3.22}) that the
statements $\pr_\zeta u\in D(T)$, $Tu_\zeta =(Au)_W$, carry over to
the statements 
\begin{equation}
\Gamma _0u\in D(L_1)\subset X_1,\quad L_1\Gamma
_0u=\pr_{Y_1}\Gamma _1u 
;\label{3.49}
\end{equation}
  this is then the boundary condition represented by $\wA$. 
Note that the condition $\Gamma _0u\in X_1$ enters
as an important part of the boundary condition, compensating for the
fact that $L_1$ acts between smaller spaces than in the case $V=Z$,
$W=Z'$.

Since $ \pr_{Y_1}\Gamma _1u= \pr_{Y_1}\Lambda _\frac12(\nu_1
u-P^0_{\gamma _0,\nu_1 }\gamma _0u)$,  
(\ref{3.49}) may also be written in terms of the standard traces $\gamma
_0u$ and $\nu_1 u$, as
\begin{equation}
\Lambda _{-\frac12}\gamma _0u\in D(L_1),\quad 
 \pr_{Y_1}\Lambda _\frac12\nu_1
u=(L_1+\pr_{Y_1}\Lambda _\frac12P^0_{\gamma _0,\nu_1 }\Lambda
_\frac12 \ij_{X_1\to L_2})\Lambda _{-\frac12}\gamma _0u.\label{3.50}
\end{equation}

Such formulations can likewise be pursued for $\wA-\lambda $, allowing an
extension of Theorems \ref{Theorem3.4} and \ref{Theorem3.4a}.
Let $\lambda \in \varrho (A_\gamma )$. Now $\Gamma _{0,V_\lambda
}=\Gamma _{0,V}F^\lambda _V:V_\lambda \simto X_1$, and $\Gamma _{0,W_{\bar\lambda}
}=\Gamma _{0,W}F^{\prime\bar\lambda }_W:W_{\bar\lambda }\simto Y_1$.
 The defining equation for
$T^\lambda $ is:
\begin{equation}
(T^\lambda \pr^\lambda _\zeta u,w)=((A-\lambda )u,w)\text{ for }u
\in D(\wA), w\in W_{\bar\lambda }.\label{3.51}
\end{equation}
We define
\begin{equation}
 L^\lambda _1=(\Gamma _{0,W_{\bar\lambda }}^*)^{-1}T^\lambda \Gamma _{0,V_\lambda }^{-1},\text{ with }D( L_1^\lambda )
=\Gamma _{0,V_\lambda }D(T^\lambda )=\Gamma _{0,V}D(T)=D(L_1),\label{3.48a}
\end{equation}
and then
rewrite the two sides as follows, denoting $\Gamma ^\lambda _1=\Lambda _\frac12(\nu_1 -P^\lambda _{\gamma _0,\nu_1 }\gamma
_0)$:
\begin{align}
(T^\lambda u^\lambda _\zeta ,w)&=(L^\lambda _1\Gamma _{0,V_\lambda
}u^\lambda _\zeta ,\Gamma _{0,W_{\bar\lambda }}w)_{L_2(\Sigma
  )}=(L_1^\lambda \Gamma _0u,\Gamma _0w)_{Y_1},\nonumber\\
((A-\lambda )u,w)&=(\Gamma ^\lambda _1u,\Gamma _0w)_{L_2(\Sigma )}=(\pr_{Y_1}\Gamma ^\lambda _1u,\Gamma _0w)_{Y_1}
. \nonumber
\end{align}
When $w$ runs through $W_{\bar\lambda }$, $\Gamma _0w$ runs through
$Y_1$, so we see that $\wA-\lambda $ represents the boundary condition
\begin{equation}
\Gamma _0u\in D(L_1)\subset X_1, \quad L^\lambda _1\Gamma _0u=\pr_{Y_1}\Gamma ^\lambda _1u.\label{3.55}
\end{equation}
It can also be written in terms of the standard trace maps as
\begin{equation}
\Lambda _{-\frac12} \gamma _0u\in D(L_1)\subset X_1,\quad L^\lambda_1\Lambda _{-\frac12}\gamma _0u=\pr_{Y_1}\Lambda _\frac12(\nu_1 u-P^\lambda _{\gamma _0,\nu_1 }\gamma
_0u),\quad 
\label{3.53} 
\end{equation}
where the equation can be rewritten as
\begin{equation}
 \pr_{Y_1}\Lambda _\frac12\nu_1
u=(L_1^\lambda +\pr_{Y_1}\Lambda _\frac12P^\lambda _{\gamma _0,\nu_1 }\Lambda
_\frac12 \ij_{X_1\to L_2})\Lambda _{-\frac12}\gamma _0u.\label{3.50a}
\end{equation}



Since $D(\wA-\lambda )=D(\wA)$, we have in view of (\ref{3.50}):
\begin{equation}
L_1^\lambda =L_1+\pr_{Y_1}\Lambda _\frac12(P^0_{\gamma _0,\nu_1 }
-P^\lambda _{\gamma _0,\nu_1 })\Lambda _\frac12 \ij_{X_1\to L_2},\text{ when $\lambda \in \varrho (A_\gamma )$. }\label{3.56}
\end{equation}

All this leads to:

\begin{theorem}
  
    \label{Theorem3.6a} For the second-order strongly elliptic operator
$A$ introduced above, let the closed realisation $\wA$ correspond to $T:V\to
W$ by Theorem {\rm \ref{Theorem2.1}}.

{\rm (i)} Define $X_1$, $ Y_1$ and $L_1$ by  {\rm(\ref{3.48})}; then $\wA$
represents the boundary condition  {\rm(\ref{3.49})}, more explicitly
written as  {\rm(\ref{3.50})}.

{\rm (ii)} For any $\lambda \in \varrho (A_\gamma )$, $\wA-\lambda $ 
corresponds to the operator $L^\lambda _1:X_1\to Y_1$ defined 
in {\rm (\ref{3.48a})}, and represents the boundary condition 
{\rm (\ref{3.55})}, also written as in {\rm (\ref{3.53}), (\ref{3.50a})}. Here 
$L^\lambda _1$ and $L_1$ are related by {\rm (\ref{3.56})}, and
    \begin{align}
  \operatorname{ker}(\wA-\lambda )&=K^\lambda _\gamma \Lambda_\frac12\operatorname{ker}L^\lambda_1,\nonumber\\
\operatorname{ran}(\wA-\lambda )&=\Gamma _{0,W_{\bar\lambda}}^*\operatorname{ran}L_1^\lambda+(H\ominus W_{\bar\lambda }).
  \label{3.58}
\end{align}

\end{theorem}

\begin{proof} In view of the preparations before the theorem, it
remains to account for (\ref{3.58}), which
follows by application of the various transformation maps to (\ref{2.38}).  
\end{proof}

The $M$-function in this set-up is defined from formula (\ref{2.41}),
in a similar way as in (\ref{3.30}):
\begin{align}
M_{L_1}(\lambda )
&=\Gamma _{0,V} M_{\wA}(\lambda )\Gamma _{0,W}^*
=\Lambda _{-\frac12}\gamma _V M_{\wA}(\lambda )\Gamma _{0,W}^*
\nonumber\\
&=\Lambda _{-\frac12}\gamma _V \pr_\zeta \bigl((I-(\wA-\lambda )^{-1}(\Ama-\lambda
)\bigr)A_\gamma  ^{-1}\ij_{W\to H}\Gamma _{0,W}^*\nonumber \\
& =\Lambda _{-\frac12}\gamma _0\bigl(I-(\wA-\lambda )^{-1}(\Ama-\lambda
)\bigr)A_\gamma  ^{-1}\ij_{W\to H}\Gamma _{0,W}^*.\label{3.57}
\end{align}

Here, when $\lambda \in \varrho (\wA)\cap\varrho (A_\gamma
)$, 
$-M_{L_1}(\lambda )$ is the inverse of $L_1^\lambda $.
We therefore have the following:

\begin{theorem}
  
    \label{Theorem3.6b} Once again, let $A$ be the second-order strongly elliptic operator
introduced above and let the closed realisation $\wA$ correspond to $T:V\to
W$ by Theorem {\rm \ref{Theorem2.1}}.
  %
%
%
  
   The $M$-function in this setting is 
  \begin{equation}
M_{L_1}(\lambda ) =\Lambda _{-\frac12}\gamma _0\bigl(I-(\wA-\lambda
)^{-1}(\Ama-\lambda)\bigr)A_\gamma  ^{-1}\ij_{W\to H} \Gamma _{0,W}^*,\label{3.59}
\end{equation}
  a family of bounded operators from $Y_1$ to $X_1$, depending holomorphically on $\lambda \in \varrho (\wA)$.   For $\lambda \in \varrho (\wA)\cap \varrho (A_\gamma )$,it satisfies
   \begin{equation}
M_{L_1}(\lambda )=-(L_1+\pr_{Y_1}\Lambda _\frac12(P^0_{\gamma _0,\nu_1}-P^\lambda _{\gamma _0,\nu_1 })\Lambda _\frac12 \ij_{X_1\to L_2})^{-1}.
\label{3.60}
\end{equation}
  The following resolvent formula  holds for all $\lambda \in \varrho (\wA)\cap\varrho (A_\gamma )$:
  \begin{align}
(\wA-\lambda )^{-1}&=(A_\gamma -\lambda )^{-1}-\ij_{V_\lambda \to
  H}\Gamma _{0,V_\lambda }^{-1} M_{L_1}(\lambda )(\Gamma
_{0,W_{\bar\lambda }}^*)^{-1}\pr_{W_{\bar\lambda }}\label{3.61}\\
&=(A_\gamma -\lambda )^{-1}-K^\lambda _{\gamma ,X_1}
M_{L_1}(\lambda )
(K^{\prime\bar\lambda }_{\gamma ,Y_1})^*,\nonumber
\end{align}
  where $K^\lambda _{\gamma ,X_1}:X_1\to H$ acts like the
  composition of $\Lambda _{\frac12}:X_1\to X$, $\ij_{X\to H^{-\frac12}}$ and $K^\lambda _\gamma
  :H^{-\frac12}\to H$.
\end{theorem}

\begin{proof} It
remains to show (\ref{3.61}), which follows from (\ref{2.37}).  
\end{proof}

The analysis of the realisations corresponding to operators $T:Z\to
Z'$ covers all the most frequently studied boundary conditions for 
second-order scalar elliptic operators, whereas the cases where $T$ 
acts between nontrivial subspaces of $Z$ and $Z'$ are more exotic. 
For example, the Zaremba problem, where Dirichlet resp.\ 
Neumann conditions are imposed on two closed subsets 
$\Sigma _D$ resp.\ $\Sigma _N$ of $\Sigma $ with common boundary 
and covering $\Sigma$,  leads to the  subspace 
$V=K^0_\gamma H^{-\frac12}_0(\Sigma _N)$, which presents additional 
technical difficulties not covered by the $\psi $dbo calculus. 

But when we go beyond the scalar second-order case, subspace 
situations have a primary interest; see the next section.

\subsection{Higher order operators and systems} \label{sec3.3}
 
Let us now consider systems (matrix-formed
operators)  and
higher order elliptic operators. Here one finds that subspace cases occur very
naturally and allow studies within the $\psi $dbo calculus, with much
the same flavour as in Theorems \ref{Theorem3.4}, \ref{Theorem3.4a}
and \ref{Theorem3.4b}.
For even-order operators, a
general and useful framework was worked out in  \cite{G74} --- 
normal boundary conditions
for operators acting between vector bundles --- 
which could be the point of departure for ample generalisations. 
Scalar operators are covered by a simpler analysis in \cite{G71}.
We shall here  show in detail how the analysis of \cite{G71} can be
used, and illustrate
systems cases by examples.

\begin{example}\label{Example3.7} Let $A=(A_{jk})$ be an $p\times p$-matrix of
second-order differential operators on $\Omega $, elliptic in the 
sense that the
determinant of the principal symbol is nonzero for $x\in\comega$, $\xi
\ne 0$. We here have a Green's formula like \eqref{3.3}, but where
$s_0(x)$ is a regular $p\times p$-matrix and $\bar s_0(x)$ is replaced
by $s_0(x)^*$. There is again a Dirichlet realisation $A_\gamma $, and if it
is elliptic and $0\in \varrho (A_\gamma )$, we can repeat the study of
Section \ref{sec3.2}, now for $p$-vectors. That will cover boundary conditions
of the form (\ref{3.23}), and give a somewhat abstract treatment of the more
general cases.
 
Now one can also consider boundary conditions where a Dirichlet
condition is imposed on some components of $\gamma _0u$ and a Neumann condition
is imposed on other components of $\nu _1u$. This is very simple to
explain when $s_0(x)=I$, so let us consider that case. Let $\wA$ be
determined by a boundary
condition of the type
\begin{equation}
\gamma _0u_{N_0}=0, \quad \gamma _1u_{N_1}-C\gamma _0u_{N_1}=0,\label{3.62}
\end{equation}
where $N_0$ and $N_1$ are complementing subsets of the index set
$N=\{1,2,\dots,p\}$ with $p_0$ resp.\ $p_1$ elements. Using the homeomorphism 
$
\gamma _Z: Z\simto \prod_{1\le j\le p}H^{-\frac12}(\Sigma )
$
and similar vector valued versions of the other mappings in Section
\ref{sec3.2}, one finds that realisations of boundary conditions (\ref{3.62})
correspond to operators 
\begin{equation}
 L:\prod_{ j\in N_1}H^{-\frac12}(\Sigma
)\to\prod_{ j\in N_1}H^{\frac12}(\Sigma ),\quad
L=C-\pr_{N_1}P^0_{\gamma _0,\gamma _1i}\ij_{N_1},\label{3.63}
\end{equation}
where $\pr_{N_1}$ projects $\{\varphi _j\}_{j\in N}$ onto $\{\varphi
_j\}_{j\in N_1}$, and $\ij_{N_1}$ injects vectors $\{\varphi _j\}_{j\in
N_1}$ into vectors indexed by
$N$ by supplying with zeroes at the places indexed by $N_0$. Again, ellipticity
makes the realisation regular, and there are formulas very similar to
those in Section \ref{sec3.2}:
\begin{align}
&L^\lambda =L+\pr_{N_1}(P^0_{\gamma _0,\gamma
_1}-P^\lambda _{\gamma _0,\gamma _1})\ij_{N_1}=C-\pr_{N_1}P^\lambda
_{\gamma _0,\gamma _1}\ij_{N_1},\quad M_L(\lambda )=-(L^\lambda
)^{-1},\label{3.64}\\
&(\wA-\lambda )^{-1}=(A_\gamma -\lambda )^{-1}-K^\lambda _\gamma  \ij_{N_1}
M_L(\lambda )\pr_{N_1}(K^{\prime\bar\lambda}_\gamma )^*,\nonumber
\end{align}
where $L^\lambda $ and $M_L(\lambda )$ are elliptic $\psi $do's when the boundary condition is elliptic.

Notation gets a little more complicated if $s_0$ is not in diagonal
form, or if  $\pr_{N_1}$ is replaced by
a projection onto a
$p_1$-dimensional subspace of ${\mathbb C}^p$ that varies with $x\in
\Sigma $. Then it is useful to apply vector bundle notation, as in
\cite{G74}, where fully general boundary conditions are treated.
\end{example}
\smallskip

Next, consider the case where $A$ is a $2m$-order operator
$A=\sum_{|\alpha |\le 2m}a_\alpha
(x)D^\alpha $
(scalar or matrix-formed), elliptic on $\comega$. The
Cauchy data are the boundary values
\begin{equation}
\varrho u=\{\gamma _0u,\dots,\gamma _{2m-1}u\},\label{3.65}
\end{equation}
which one can split into the Dirichlet data and the Neumann data
\begin{equation}
\gamma u= \{\gamma _0u,\dots,\gamma _{m-1}u\},\quad \nu u=\{\gamma _mu,\dots,\gamma _{2m-1}u\}.\label{3.66}
\end{equation}
There is a Green's formula for $u,v\in H^{2m}(\Omega )$:
\begin{equation}
(Au,v)-(u,A'v)= (\chi  u, \gamma v)-(\gamma u, \chi '
 v),\quad \chi u={\cal A}_1\nu u,\quad \chi 'v={\cal A}'_1\nu
v+{\cal A}'_0\gamma v,\label{3.67}
\end{equation}
where 
${\cal A}_1$ and  ${\cal A}'_1$ (both
invertible) and ${\cal A}'_0$ are suitable $m\times m$ matrices of differential operators over
$\Sigma $. Here $\chi =\{\chi _0,\dots,\chi _{m-1}\}$
with $\chi _j$ of order $2m-j-1$, and the trace maps are continuous,
\begin{equation}
\gamma :Z_\lambda ^s(A)\to \prod_{0\le j\le m-1}H^{s-j-\frac12}(\Sigma
),\quad \chi  :Z_\lambda ^s(A)\to \prod_{0\le j\le m-1}H^{s-2m+j+\frac12}(\Sigma
),\label{3.68}
\end{equation}
for $s\in {\mathbb R}$. 
 If the Dirichlet problem is uniquely
solvable, one can again use $A_\gamma $ as reference operator, and set
\begin{equation}
P^\lambda _{\gamma ,\chi }=\chi K^\lambda _\gamma ;\quad 
P^{\prime\bar\lambda }_{\gamma ,\chi '}=\chi 'K^{\prime\bar\lambda }_\gamma ;
 \label{3.70}
\end{equation}
with $K^\lambda _\gamma $ resp.\ $K^{\prime\bar\lambda }_\gamma$
denoting the Poisson operator solving
\begin{equation}
(A-\lambda )u=0\text{ resp.\ }(A'-\bar\lambda )u=0\text{ in }\Omega
,\quad \gamma u= \varphi .\label{3.71} 
\end{equation}
Now
\begin{equation}
\Gamma =\chi -P^0_{\gamma ,\chi }\gamma \;\text{ and }\; \Gamma '=\chi
'-P^{\prime 0}_{\gamma ,\chi '}\gamma 
\end{equation}
are defined as continuous maps from $D(\Ama)$, resp.\ $D(\Ama')$ to
$\prod _{j<m }H^{j+\frac12}(\Sigma )$, and 
there is a generalisation of (\ref{3.67}) valid for all $u\in
D(\Ama)$, $v\in D(\Ama')$:
\begin{equation}
(Au,v)-(u,A'v)= (\Gamma   u, \gamma v)_{\{j+\frac12\},\{-j-\frac12\}}-(\gamma u, \Gamma '
 v)_{\{-j-\frac12\},\{j+\frac12\}}\label{3.69}
\end{equation}
(where $(\cdot,\cdot)_{\{s_j\} ,\{-s_j\} }$ indicates the duality pairing
between $\prod _{j<m }H^{s_j}(\Sigma )$ and $\prod_{j<m }
H^{-s_j}(\Sigma )$). 

It should be noted that $P^\lambda _{\gamma ,\chi }$ is a {\it mixed-order
system}:
$$
P^\lambda _{\gamma ,\chi }=(P_{jk})_{j,k=0,\dots,m-1};  \quad P_{jk}\text{ of order }2m-1-j-k,
$$
and the principal symbol and possible ellipticity is defined accordingly. Such systems (with
differential operator entries) were first considered by Douglis and
Nirenberg \cite{DN55} and by Volevich
\cite{V63}; in the elliptic case they are called Douglis-Nirenberg elliptic. 

Formula (\ref{3.69}) can be turned into a boundary triplet formula with
${\cal H=\cal K}=L_2(\Sigma )^m$
by composition of $\Gamma $ and $\Gamma '$ with a symmetric $\psi $do
defining the norms:
\begin{equation}
\Theta =(\delta _{kl}\Lambda
_{k+\frac12})_{k,l=0,\dots,m-1}:\prod_{0\le j<m } 
H^{j+\frac12}(\Sigma )\simto L_2(\Sigma )^{m},\label{3.72}
\end{equation}
and composition of $\gamma $ with $\Theta ^{-1}$, setting 
\begin{equation}
\Gamma _1=\Theta \Gamma ,\quad \Gamma '_1=\Theta \Gamma ',\quad 
\Gamma _0=\Theta ^{-1}\gamma =\Gamma '_0;\label{3.73}
\end{equation}
 this leads to a word-for-word generalisation of Proposition 3.1,
with 
\begin{equation}
(Au,v)_{L_2(\Omega )}-(u,A'v)_{L_2(\Omega )}=(\Gamma _1 u,\Gamma '_0v)_{L_2(\Sigma
)^m}-(\Gamma _0u, \Gamma _1'v)_{L_2(\Sigma )^m},
\label{3.73a}
\end{equation}
for $u\in D(\Ama)$, $ v\in D(\Ama')$.

The straightforward continuation of what we did in Section 3.2 
is the study of boundary conditions of the type 
$
\chi  u=C \gamma u,
$
that are the ones obtained when $\wA$ corresponds to $T:Z\to Z'$. When
$\wA$ corresponds to $T:V\to W$, we find that it represents a boundary
condition similar to (\ref{3.49}). Since the proofs in Section 3.2
generalise immediately to this situation, we can state:

\begin{corollary}\label{Corollary3.8}
Theorems {\rm \ref{Theorem3.4}}, {\rm \ref{Theorem3.4a}} and {\rm \ref{Theorem3.4b}} extend to the present situation for
elliptic operators of order $2m$, when
$H^{-\frac12}(\Sigma )$ is replaced by $\prod_{j<m } H^{-j-\frac12}(\Sigma )$, and $\gamma
_0$ and $\nu _1$ are replaced by $\gamma $ and $\chi $.

Theorems {\rm \ref{Theorem3.6a}} and {\rm \ref{Theorem3.6b}} likewise extends when, moreover, $L_2(\Sigma )$
is replaced by $L_2(\Sigma )^m$ and {\rm (\ref{3.73a})} is used.

\end{corollary}

Whereas the statements in the case of general subspaces $V$ and $W$
will
be somewhat abstract, there are now also formulations where
the subspaces are represented by products of Sobolev spaces, and the
operators 
belong to the $\psi $dbo calculus. We shall demonstrate this 
on the basis of the treatment of normal boundary value problems 
in \cite{G71}, which we now recall. 
Denote
\[
M=\{0,\dots, 2m-1\},\quad M_0=\{0,\dots, m-1\},\quad M_1=\{m,\dots, 2m-1\}.
\]
A general normal boundary condition is given as 
\begin{equation}
\gamma _ju+\sum_{k<j}B_{jk}\gamma _ku=0,\quad j\in J,\label{3.74a}
\end{equation}
where $J$ is a subset of $M$ with $m$ elements and the $B_{jk}$ are
differential operators on $\Sigma $ of order $j-k$. (It is called
normal since the highest-order trace operators $\gamma _j$ have
coefficient 1.) Let $K=M\setminus
J$, and set
\begin{align}
J_0&=J\cap M_0,\quad J_1=J\cap M_1,\quad
K_0=K\cap M_0,\quad K_1=K\cap M_1,\nonumber\\
\gamma _{J_0}&=\{\gamma _j\}_{j\in J_0}, 
 \quad \nu _{J_1}=\{\gamma
_j\}_{j\in J_1},\quad \gamma _{K_0}=\{\gamma _j\}_{j\in K_0}, 
 \quad \nu _{K_1}=\{\gamma
_j\}_{j\in K_1},\nonumber
\end{align}
then (\ref{3.74a}) can be reduced to the form
\begin{equation}
\gamma _{J_0}u=F_0\gamma _{K_0}u,\quad \nu _{J_1}u=F_1\gamma
_{K_0}u+F_2\nu _{K_1}u,\label{3.74b} 
\end{equation}
with suitable matrices of differential operators $F_0, F_1, F_2$.
To reformulate this in terms of $\{\gamma ,\chi \}$,
 we use the convention for reflected sets:
\begin{equation}
\quad N'=\{j\mid 2m-1-j\in N\},\label{3.79}
\end{equation}
considered as an ordered subset of $M$. Then $\chi =\{\chi _j\}_{j\in
M_0}$ splits into $\chi =\{\chi _{{J_1}'},\chi _{{K_1}'}\}$, where 
\[
 \chi  _{{J_1}'}=\{\chi  _j\}_{j\in {J_1}'}, \quad  \chi  _{{K_1}'}=\{\chi  _j\}_{j\in {K_1}'}.
\]
We can now reformulate (\ref{3.74b})
as 
\begin{equation}
\gamma _{J_0}u=F_0\gamma _{K_0}u,\quad \chi  _{{J_1}'}u=G_1\gamma
_{K_0}u+G_2\chi  _{{K_1}'}u,\label{3.74d}
\end{equation}
where $G_1$ and $G_2$ are matrices of differential operators derived
from those in (\ref{3.74b}) and the coefficients in Green's formula
(more details in \cite{G71}).

When $\wA$ is the realisation of $A$ determined by the boundary
condition (\ref{3.74d}) (i.e., $D(\wA)$ consists of the $u\in D(\Ama)$
satisfying (\ref{3.74d})), and ellipticity holds, then $D(\wA)\subset
H^{2m}(\Omega )$, and the adjoint realisation
$\wA^*$ is determined by the likewise elliptic boundary condition 
\begin{equation}
\gamma _{{K_1}'}v=-G_2^*\gamma _{{J_1}'}v,\quad \chi '_{K_0}v=G_1^*\gamma _{{J_1}'}v-F_0^*\chi '_{J_0}v.\label{3.74e}
\end{equation}
Setting
\[
\Phi =\begin{pmatrix} I_{K_0K_0}\\ F_0\end{pmatrix},\quad \Psi =\begin{pmatrix} I_{{J_1}'{J_1}'}\\ -G_2^*\end{pmatrix},
\]
we have that
\[
\gamma D(\wA)\subset X=\Phi \big(\prod_{k\in K_0
}H^{-k-\frac12}(\Sigma )\big),\quad \gamma D(\wA^*)\subset Y=\Psi \big(\prod_{j\in {J_1}' } H^{-j-\frac12}(\Sigma )\big).
\]
Here $X$ is the graph of $F_0$ and naturally homeomorphic to its
``first component'' $\prod_{k\in K_
0 }H^{-k-\frac12}(\Sigma )$, and
similarly $Y$ is homeomorphic to $\prod_{j\in {J_1}' }
H^{-j-\frac12}(\Sigma )$. The operator $T:V\to W$ that $\wA$
corresponds to by Theorem 2.1 
carries over to an operator $L:X\to
Y^*$ by use of $\gamma $, and this is further reduced to an operator 
\[
L_1:\prod_{k\in K_
0 }H^{-k-\frac12}(\Sigma )\to \prod_{j\in {J_1}' }
H^{j+\frac12}(\Sigma ),\quad L_1=G_1-\Psi ^*P^0_{\gamma ,\chi }\Phi ,
\]
with domain $D(L_1)=\prod_{k\in K_
0 }H^{2m-k-\frac12}(\Sigma )$.
This representation is used in \cite{G71} to find criteria for the operator to be
m-accretive, and the ideas are further pursued in \cite{G74} for
systems of operators 
(where vector bundle notation is needed).

For the present study of resolvents, we now find that $\wA-\lambda
$ can be represented by 
\begin{equation}
L^\lambda _1=G_1-\Psi ^*P^\lambda  _{\gamma ,\chi }\Phi :\prod_{k\in K_
0 }H^{-k-\frac12}(\Sigma )\to \prod_{j\in {J_1}' }
H^{j+\frac12}(\Sigma ),\quad D(L^\lambda _1)=\prod_{k\in K_
0 }H^{2m-k-\frac12}(\Sigma ),\label{3.74f}
\end{equation}
 when $\lambda \in \varrho (A_\gamma )$.
The corresponding $M$-function and Kre\u\i{}n formula are: 
\begin{align}
M(\lambda )=-(G_1&-\Psi ^*P^\lambda _{\gamma ,\chi }\Phi )^{-1}:\prod_{j\in {J_1}' }
H^{j+\frac12}(\Sigma )\to \prod_{k\in K_
0 }H^{-k-\frac12}(\Sigma ) ,\nonumber\\
(\wA-\lambda )^{-1}&=(A_\gamma -\lambda )^{-1}-K^\lambda _\gamma  \Phi 
M(\lambda )(K^{\prime\bar\lambda}_\gamma \Psi )^*.
\label{3.74g}
\end{align}
When $\lambda \in \varrho (\wA)\cap \varrho (A_\gamma )$, $M(\lambda )$ extends
holomorphically to $\varrho (\wA)$ (note that the spectrum of $\varrho
(A_\gamma )$ is discrete in this case). As a mixed-order operator
with entries of order $2m-1-j-k$ ($k\in K_0$, $j\in {J_1}'$),
$L^\lambda _1$ is Douglis-Nirenberg elliptic, and $M(\lambda )$ is so
in the opposite direction.

The two functions $M(\lambda )$ and $L^\lambda _1$ together give a
tool to analyse the spectral properties of $\wA$ in terms of $\psi
$do's on $\Sigma $, $M(\lambda )$ being holomorphic on $\varrho (\wA)$
and $L^\lambda _1$ containing information on null-spaces and
ranges. We have hereby obtained:

\begin{theorem}
Let $A$ be a $2m$-order elliptic differential operator with
coefficients in $C^\infty (\comega )$, and let $\wA$ be the
realisation defined by the normal boundary condition {\rm
  (\ref{3.74a})}, reformulated as {\rm(\ref{3.74d})}; assume that the
boundary problem is elliptic. Let $A_\gamma $ be the Dirichlet
realisation, assumed elliptic and invertible. For $\lambda \in \varrho
(A_\gamma )$, the operator corresponding to $\wA-\lambda $ by
Corollary {\rm\ref{Corollary2.3}} carries over to {\rm (\ref{3.74f})}.
The associated $M$-function and Kre\u\i{}n resolvent formula are described in
{\rm (\ref{3.74g})}; $M(\lambda )$ extends to an
operator family holomorphic in $\lambda \in \varrho (\wA)$.
\end{theorem}

When $A$ is scalar of order $2m$, this analysis covers {\it all} the elliptic problems to which Seeley's resolvent
construction \cite{Se66} applies, for it is known that normality of the
boundary condition is necessary for the parameter-ellipticity required
there (cf.\ e.g.\ \cite[Section 1.5]{G96}). When $p\times p$-systems are considered, each line in
(\ref{3.74a})  can moreover be
composed with a multiplication map to the sections of a subbundle of $\Sigma \times {\mathbb C}^p$,
as indicated in Example \ref{Example3.7}; here normality means surjectiveness of the
coefficient of the highest normal derivative $\gamma _j$, for each $j$. For such cases, the treatment
can be based on \cite{G74}.

\begin{example}\label{Example3.9} For a simple example, consider the biharmonic
operator $A=\Delta ^2$, which has the Green's formula
\begin{equation}
(\Delta ^2u,v)-(u,\Delta ^2v)=(\chi u,\gamma v)-(\gamma u,\chi
v),\quad \chi u=\begin{pmatrix} -\gamma _1\Delta u\\ \gamma _0\Delta u\end{pmatrix};\label{3.74}
\end{equation}
note that linear conditions on $\gamma _0u$, $\gamma _1u$, $\gamma _2u$ and 
$\gamma _3u$
can be written as conditions on $\gamma _0u$, $\gamma _1u$, $\gamma
_0\Delta u$ and $\gamma _1\Delta u$, hence on $\gamma u$ and $\chi u$. The
operator $P^\lambda _{\gamma ,\chi }$ is a $\psi $do from
$H^{s-\frac12}\times H^{s-\frac32}$ to 
$H^{s-\frac72}\times H^{s-\frac52}$
(note the reverse order of the Sobolev exponents in the target
space). 

Let $\wA$
be defined by
an elliptic boundary condition with a first-order differential
operator $C$,
\begin{equation}
\gamma _0u=0,\quad \gamma _0\Delta u=C\gamma _1u;\label{3.75}
\end{equation}
then $\wA^*$ represents the likewise elliptic boundary condition
\begin{equation}
\gamma _0u=0,\quad \gamma _0\Delta u=C^*\gamma _1u,\label{3.76}
\end{equation}
and the domains of $\wA$ and $\wA^*$ are contained in $H^4(\Omega )$. For the
corresponding operator $T:V\to W$ according to Theorem 2.1,
$V=W=K^0_\gamma (\{0\}\times H^{-\frac32})$. This carries over to the
 operator
\begin{equation}
L=C-\pr_2P^0_{\gamma ,\chi }\ij_2
:H^{-\frac32}\to H^{\frac32},\label{3.77}
\end{equation}
where $\pr_2=\begin{pmatrix} 0&1\end{pmatrix}$ and $\ij_2=\begin{pmatrix}
0\\1\end{pmatrix}$. Thanks to the
ellipticity, $D(L)=H^{\frac52} $, and $L$ is an elliptic $\psi $do of
order 1.
We here find that
\begin{align}
&L^\lambda =L+\pr_2(P^0_{\gamma ,\chi }-P^\lambda _{\gamma ,\chi
})\ij_2=C-\pr_2P^\lambda _{\gamma ,\chi
}\ij_2,\quad M_L(\lambda )=-(L^\lambda )^{-1},\label{3.78}\\
&(\wA-\lambda )^{-1}=(A_\gamma -\lambda )^{-1}-K^\lambda _\gamma  \ij_2
M_L(\lambda )\pr_2(K^{\prime\bar\lambda}_\gamma )^*.\nonumber
\end{align}

\end{example}
\medskip

It is also possible to take another realisation $A_\beta $ than the Dirichlet
realisation as reference operator, preferably one defined by an
elliptic boundary condition, as in \cite{G68}. For strongly elliptic 
operators, using $A_\gamma $ as reference
operator has the advantage
that semiboundedness properties are preserved in correspondences
between $\wA$ and $T$, 
see \cite{G71,G74}.

There do exist even-order operators for which the Dirichlet problem is
not elliptic; a well-known example is the operator $-\Delta
I+2\operatorname{grad}\divg$ on subsets of ${\mathbb R}^2$, studied by Bitsadze.

\begin{example}\label{Example3.10} The operator $A$ need not be of even order. For
example, first-order $p\times p$-systems, such as Dirac operators, have received
much attention. In this case the Cauchy data are $\{\gamma
_0u_1,\dots,\gamma _0u_p\}$. Elliptic boundary conditions require 
$p$ to be even. In ``lucky'' cases, one can get an elliptic boundary value
problem by imposing the vanishing of half of the boundary values; this
will then give a reference problem, allowing the discussion of other
realisations. More systematically, one can impose the condition $\Pi
_+\gamma _0u=0$ for a certain $\psi $do projection $\Pi _+$ over the boundary
(the Atiyah-Patodi-Singer condition), which defines a Fredholm
realisation. If it is invertible, it can be used as a reference operator. 

\end{example}

Similar considerations can be worked out for elliptic operators on
suitable unbounded domains, and on manifolds. More specifically, the
 calculus of $\psi $dbo's
is extended in \cite{G96} to the manifolds called ``admissible'' there; they
have finitely many unbounded ends with good control over coordinate
diffeomorphisms. They include complements in
${\mathbb R}^n$ as well as 
$\crnp$ of smooth bounded domains.

\begin{example}\label{Example3.11a} We shall here make concrete the considerations in
Example \ref{Example3.9} for the biharmonic operator on the half-space
 $$
\rnp=\{x=\{x',x_n\}\in{\mathbb R}^n\mid x_n>0\},
 $$
where we denote $\{x_1,\dots,x_{n-1}\}=x'$. This constant-coefficient
case can be viewed as a model for variable-coefficient cases, giving
an example of the pointwise symbol calculations entering in the $\psi
$dbo theory. (A detailed introduction to the $\psi $dbo calculus is
found e.g.\ in \cite{G08}.)

It is well-known that since $\Delta ^2$ is
symmetric nonnegative, the resolvent of the Dirichlet realisation
in $L_2(\rnp)$ exists for $\lambda \in {\mathbb C}\setminus \crp$. We
shall write
$$
\lambda =-\mu ^4, \quad \mu \in V_{\pi /4},\text{ where }V_{\theta }=\{z\in {\mathbb C}\setminus\{0\}\mid  |\arg z|<\theta \}.
$$

We fix a $\lambda _0<0$, then the realisations of $A=\Delta ^2-\lambda
_0$ fit into the general
set-up, with the Dirichlet realisation $A_\gamma
$ of $\Delta ^2-\lambda _0$ as reference operator. However, for simplicity in formulas (avoiding
addition and subtraction of $\lambda _0$), we keep the
parameter $\lambda $ for the operator families defined relative to
$\Delta ^2-\lambda $.

In the following calculations, one can think of a $\mu >0$; 
the considerations extend
holomorphically to $V_{\pi /4}$. To find the Poisson
operator $K^\lambda _\gamma $ solving the problem
$$
(\Delta ^2-\lambda )u(x',x_n)=0\text{ on }\rnp,\quad u(x',0)=\varphi
_0(x'),\quad \partial_n u(x',0)=\varphi _1(x'),
$$
 we perform a Fourier transformation in the $x'$-variable,
and then have to find solutions of
\begin{align}
((|\xi '|^2-\partial_n^2)^2+\mu ^4)\acute u(\xi ',x_n)&=0\text{ for }x_n>0,\nonumber\\
\acute u(\xi ',0)&=\hat\varphi _0(\xi '),\nonumber\\\partial_n\acute u(\xi ',0)&=\hat\varphi _1(\xi '),
\label{3.81}
\end{align}
that are in ${\cal S}(\crp)=r^+{\cal S}({\mathbb R})$ (the restriction to
$\rp$ of the Schwartz space of rapidly decreasing functions on ${\mathbb
R}$).
Write
\begin{align}
((|\xi '|^2-\partial_n^2)^2+\mu ^4)&=(|\xi '|^2+i\mu
^2-\partial_n^2)(|\xi '|^2-i\mu ^2-\partial_n^2)\nonumber\\
&=(\sigma _++\partial_n)(\sigma _+-\partial_n)(\sigma
_-+\partial_n)(\sigma _--\partial_n),\nonumber\\
\sigma _+&=(|\xi '|^2+i\mu ^2)^{\frac12},\quad \sigma _-=(|\xi
'|^2-i\mu ^2)^{\frac12},
\end{align}
where $z^\frac12$ is defined for $z\in V_\pi $ to be real positive for
$z\in \rp$; note that $\mu ^2\in V_{\pi /2}$ so that $|\xi '|^2\pm i\mu
^2\in V_\pi $, and $\operatorname{Re}\sigma _\pm>0$. Then
the general solution in $\cal S(\crp)$ of the first line in (\ref{3.81}) is
$$
v(x_n)=c_1e^{-\sigma _+x_n}+c_2e^{-\sigma _-x_n}.
$$
It is adapted to the boundary conditions by solution of
$$
c_1+c_2=\hat\varphi _0(\xi '),\quad -\sigma _+c_1-\sigma _-c_2=\hat\varphi _1(\xi '),
$$
with respect to $(c_1,c_2)$; this gives
$$
\acute u(\xi ',x_n)=\tfrac1{\sigma _+-\sigma _-}\begin{pmatrix} e^{-\sigma
_+x_n}&e^{-\sigma _-x_n}\end{pmatrix}
 \begin{pmatrix} -\sigma _-&-1\\\sigma _+&1\end{pmatrix}
\begin{pmatrix} \hat\varphi _0\\ \hat\varphi _1\end{pmatrix}\equiv \tilde k^\lambda _\gamma \begin{pmatrix} \hat\varphi _0\\ \hat\varphi _1\end{pmatrix}.
$$
Here $\tilde k^\lambda _\gamma (\xi ',x_n)$ is the so-called symbol-kernel of the
Poisson operator $K^\lambda _\gamma $; it acts like
$$
K^\lambda _\gamma \begin{pmatrix} \varphi _0\\\varphi _1\end{pmatrix} =(2\pi
)^{1-n}\int _{{\mathbb R}^{n-1}}e^{ix'\cdot\xi '}\tilde k^\lambda _\gamma (\xi ',x_n)\begin{pmatrix} \hat\varphi
_0(\xi ')\\\hat\varphi _1(\xi ')\end{pmatrix}\, d\xi ',\label{3.82}
$$
and maps $H^{s-\frac12}({\mathbb R}^{n-1})\times H^{s-\frac32}({\mathbb
R}^{n-1})$ to $\{u\in H^{s}({\mathbb R}^{n}_+)\mid (A-\lambda )u=0\}$ for
all $s\in{\mathbb R}$. (In variable-coefficient cases, $\tilde k$ would
moreover depend on $x'$.) 

The symbol $p^\lambda (\xi ')$ of the Dirichlet-to-Neumann operator $P^\lambda _{\gamma
,\chi }$ with 
$\chi $ chosen as in (\ref{3.74}) is found by the calculation 
\begin{align}
p^\lambda (\xi ' )&=\gamma _0\begin{pmatrix} \partial_n(|\xi '|^2-\partial_n^2)\\
-|\xi '|^2+\partial_n^2 \\
\end{pmatrix}\tilde k^\lambda _\gamma (\xi ',x_n)=\frac1{\sigma _+-\sigma _-}\begin{pmatrix} -i\sigma _+\mu ^2 &  i\sigma _-\mu ^2\\i\mu ^2 & -i\mu ^2
\end{pmatrix} \begin{pmatrix} -\sigma _-&-1\\\sigma _+&1\end{pmatrix}
\nonumber\\
&=\frac{i\mu ^2}{\sigma _+-\sigma _-}\begin{pmatrix}{2\sigma _+\sigma _-}
&\sigma _++\sigma _- \\-(\sigma _++\sigma _- )&-2
\end{pmatrix}\nonumber\\
&=\tfrac14(\sigma _++\sigma _-)\begin{pmatrix}{2\sigma _+\sigma _-}
&\sigma _++\sigma _-\\-(\sigma _++\sigma _-) &-2 \end{pmatrix}
=\begin{pmatrix} p^\lambda _{00}&p^\lambda _{01}\\p^\lambda
  _{10}&p^\lambda _{11}\end{pmatrix}; \label{3.83} 
\end{align}  
here we used that 
$(\sigma _+-\sigma _-)^{-1}=(\sigma _++\sigma _-)/(\sigma
_+^2-\sigma _-^2)= (\sigma _++\sigma _-)/(4i\mu ^2)$.

The operator $P^\lambda _{\gamma ,\chi }$ is $\Op(p^\lambda )$, where 
we use the notation for $\psi $do's $$
\Op(q)f=(2\pi )^{1-n}\int_{{\mathbb R}^{n-1}}e^{ix'\cdot\xi '}q(x',\xi
')\hat f(\xi ')\,d\xi '.
$$

Let us consider a model boundary condition (\ref{3.75}) for $\Delta ^2-\lambda _0$ 
with a first-order
differential operator $C=b_1\partial_1+\dots+b_{n-1}\partial_{n-1}$
with constant coefficients; it defines the realisation $\wA$, and $C$ has
symbol $c(\xi ')=ib\cdot \xi '$.

 Then the operators $L^{\lambda _0}$ and $L^\lambda $ defined similarly to
(\ref{3.77}) and (\ref{3.78}) are the
$\psi $do's  with symbol
\begin{align*}
l^{\lambda _0}(\xi ')&=c(\xi ')-p^{\lambda_0}_{11}(\xi '),\text{
  resp.\ }\\
l^\lambda ( \xi ' )&=l^{\lambda_0}(\xi ')+p^{\lambda_0}_{11}(\xi ')-p^\lambda
_{11}(\xi ' )=c(\xi ')-p^{\lambda }_{11}(\xi ')=c(\xi
')+\tfrac12(\sigma _++\sigma _-). 
\end{align*}
When $l^\lambda $ is invertible for all $\xi '$,  $M_L(\lambda )$ is the $\psi $do with symbol 
\begin{equation}
m(\xi ',\lambda )=-\big(c(\xi ')+\tfrac12(\sigma _++\sigma _-)\big)^{-1}=-\big(c(\xi ')+\tfrac12(|\xi '|^2+i\mu ^2)^{\frac12}+\tfrac12(|\xi
'|^2-i\mu ^2)^{\frac12}\big)^{-1}.\label{3.84}
\end{equation}

This shows how $M_L(\lambda )$ is found. As an additional observation,
we remark that when $b$ is real nonzero, then $C^*=-C\ne 0$, so $\wA$ is non-selfadjoint,
cf.\ (\ref{3.76}). However, $M_L(\lambda )$ is well-defined for all $\lambda \in
{\mathbb C}\setminus \crp$ (since $c(\xi ')\in i{\mathbb R}$). On the other
hand, $L^{\lambda _0}$ has numerical range $\nu (L^{\lambda _0})$ approximately
equal to the sector $V=\{z\in{\mathbb C}\mid |\operatorname{Im}z|\le
|b|\operatorname{Re}z\}$, and $\nu (L^{\lambda _0})$ is contained in the numerical range
of $\wA$, by \cite[Th. III 4.3]{G68}. This gives an example where $\wA$ has a large
numerical range outside the spectrum.
\end{example}

 \begin{remark}\label{Remark3.11}  One can ask whether the
   considerations extend to Douglis-Nirenberg elliptic systems
   (systems  of
   mixed order). But it may not be easy. Consider for example a $2\times 2$-system
\begin{equation}
Aw=\begin{pmatrix} A_{11}& A_{12}\\A_{21}&A_{22}\end{pmatrix}\begin{pmatrix} u\\v\end{pmatrix}
\label{3.80}
\end{equation}
on a smooth bounded open set $\Omega \subset {\mathbb R}^n$,
where $u$ and $v$ are scalar, and $A_{ij}$ is of order $4-i-j$.
There is a Green's formula for $A$:
\[
(Aw,w')-(w,A^*w')=(\kappa w,\gamma _0u')-(\gamma _0u,\kappa 'w'),
\]
 where $\kappa w= \chi u +b_1\gamma _0v$, $ \kappa
'w'=
\chi 'u'+b_2\gamma _0 v'$,
with first-order trace operators $\chi $ and $\chi '$; $\{\gamma
_0u,\kappa w\}$ are the {\it reduced Cauchy data} according to
\cite{G77, GG77}.
Assume that the principal symbol is uniformly positive definite; in
particular, the function $A_{22}(x)\ge c>0$, the symbol $a^0_{11}(x,\xi )\ge c|\xi
|^2$ and the determinant $a^0_{11}(x,\xi )A_{22}(x)-a^0_{12}(x,\xi
)a^0_{21}(x,\xi )\ge c|\xi |^2$.
Then the Dirichlet problem for $A$ (with boundary condition $\gamma _0u=0$) 
is well-posed, with domain
$
D(A_\gamma )=(H^2(\Omega )\cap H^1_0(\Omega ))\times H^1(\Omega )$.
 
When $0\in\varrho (A_\gamma )$, there is the following parametrization of null-spaces of $A$ (cf.\ also
\cite[Th. 2.11]{G77}):
\[
\gamma _0\pr_1 : Z^{s,s-1}_0(A)=\{w\in H^{s}(\Omega )\times
H^{s-1}(\Omega )|Aw=0\}\simto H^{s-\frac12}(\Sigma ),
\]
all $s\in{\mathbb R}$; but this does not in general lead to a nice parametrization of 
\[
Z^0_0(A)=\{w\in L_2(\Omega )^2|Aw=0\},
\]
which would be required to get a good interpretation of the abstract
theory for $A$ as an operator in $L_2(\Omega )^2$. 

This regularity problem is not present in the one-dimensional
situation, where the maximal domain is $H^2\times H^1$; an example is
considered in Section 4.2.

In the example in Section 4.1  we  consider an $n$-dimensional case of
(\ref{3.80}), where the off-diagonal terms are of order 0; this
allows an easier parametrization of the null-space.
\end{remark}

\section{Examples}\label{sec4}

We have seen that the family $M_{\wA}(\lambda )$ is holomorphic on
$\varrho (\wA)$ so that 
\begin{equation}
\sigma (\wA)\supset \{\lambda \in{\mathbb C}\mid M_{\wA}(\lambda )\text{
singular at }\lambda \}.\label{4.1}
\end{equation}
Among the singular points, we have very good control of those outside
$\varrho (A_\beta )$ by Corollary \ref{Corollary2.11} and \eqref{2.38}, for the null-space and range of
$\wA-\lambda $ are fully clarified by the same concepts for $T^\lambda
$
(a holomorphic operator family on $\varrho (A_\beta )$ which 
is homeomorphic to the inverse of
$M_{\wA}(\lambda )$ on $\varrho (\wA)\cap \varrho (A_\beta )$); it
gives the information:
\begin{equation}
\sigma (\wA)\cap \varrho (A_\beta )=\{\lambda \in \varrho (A_\beta
)\mid \operatorname{ker} T^\lambda \ne \{0\}\text{ or
}\operatorname{ran}T^\lambda \ne W_{\bar\lambda } \}.\label{4.2}
\end{equation}
 So the only spectral points of $\wA$ whose spectral nature may not 
be controlled by $M_{\wA}(\lambda )$ and  $T^\lambda $ are 
those that lie in $\sigma
(A_\beta )$. For many scalar equations it has long been known that the $M$-function allows full control of the spectrum. However, when considering systems, it is easy to see that uncontrolled  points may exist by considering equations which are decoupled. In \cite{BMNW08}, an example involving a coupled system of ODEs was
given where $M_{\wA}$ was regular at a point $\lambda _0$ belonging to
the essential spectrum of $\wA$ (and of $A_\beta $).
We shall here show a similar phenomenon for PDEs and for a system of ODEs with first order off-diagonal entries.

\subsection{PDE counterexample}\label{sec4.1}

Consider the $2\times 2$ matrix-formed operator
\begin{equation}
A=\begin{pmatrix} A_0&a(x)\\b(x)&c(x)\end{pmatrix},\label{4.3}
\end{equation}
acting on 2-vector functions $w=\{u,v\}$ on $\Omega $, such that $A_0$ is a second-order
elliptic operator as studied in Section \ref{sec3.2}  and $a,b,c\in C^\infty
(\comega)$. The set $\operatorname{ran}c$ is a compact
subset of ${\mathbb C}$. We assume that it has a connected component
$K$ with more than one point, that ${\mathbb C}\setminus
(\operatorname{ran}c)$ is connected, and that 
\begin{equation}
a(x)b(x)\text{ vanishes on }\supp c .\label{4.3a}
\end{equation}
The maximal and minimal operators are 
\begin{equation}
A_{\max}=\begin{pmatrix} A_{0,\max}&a(x)\\b(x)&c(x)\end{pmatrix},\quad A_{\min}=\begin{pmatrix} A_{0,\min}&a(x)\\b(x)&c(x)\end{pmatrix},\label{4.4}
\end{equation}
with 
\begin{equation}
D(\Ama)=D(A_{0,\max})\times L_2(\Omega ),\quad D(\Ami)=D(A_{0,\min})\times L_2(\Omega ),\label{4.5}
\end{equation}
and there is the Green's formula
$$
\left(A\begin{pmatrix} u\\v\end{pmatrix}, \begin{pmatrix} u'\\v'\end{pmatrix}\right)-\left(\begin{pmatrix} u\\v\end{pmatrix},A' \begin{pmatrix} u'\\v'\end{pmatrix}\right)=(\nu _1u,\gamma _0u')-
(\gamma _0u, \nu _1'u'+{\cal A}'_0\gamma _0u')_{L_2(\Sigma )},
$$
with notation as in (\ref{3.3})\,ff.

\begin{proposition}
Let $A_\gamma $ be the Dirichlet realisation defined by the Dirichlet
condition $\gamma _0u=0$.
 
{\rm (i)}  $A_\gamma $ is lower semibounded with domain
$(H^2(\Omega )\cap H^1_0(\Omega ))\times L_2(\Omega )$, and the
spectrum is contained in a half-space 
$\{\operatorname{Re}\lambda \ge \alpha \}$.

{\rm (ii)} The operator $A_\gamma $ has a non-empty essential spectrum, namely 
\begin{equation}
\sigma _{\operatorname{ess}}(A_\gamma )=\operatorname{ran}c.\label{4.6}
\end{equation}

{\rm (iii)} Outside the
essential spectrum, the spectrum is discrete, consisting of
eigenvalues of finite multiplicity.
\end{proposition}

\begin{proof}
{\rm (i)} follows from standard results for elliptic operators and the fact that
 the adjoint of $A_\gamma$ is
the Dirichlet realisation of $A'$ with similar properties. Hence, the
spectrum is contained in a half-space 
$\{\operatorname{Re}\lambda \ge \alpha \}$.

{\rm (ii)} 
follows from Geymonat-Grubb \cite{GG77}, where it was shown that the essential spectrum of the
realisation of a mixed-order
system $A$ defined by a differential boundary condition $\beta u=0$
consists
exactly of the points $\lambda \in{\mathbb C}$ where ellipticity of
$\{A-\lambda ,\beta \}$ fails.    The current operator $A$ is a mixed-order
system with orders 2 and 0 for the diagonal terms, order 1 for the
off-diagonal terms (to fit with the rules of Douglis, Nirenberg and
Volevich for mixed-order systems). The principal symbol of $A-\lambda $ is
\begin{equation}
\begin{pmatrix} a^0(x,\xi )&0\\0&c(x)-\lambda \end{pmatrix},\label{4.7}
\end{equation}
so ellipticity of the Dirichlet problem for $A-\lambda $ fails
precisely when $\lambda \in \operatorname{ran}c$.

For {\rm (iii)},  we note that the resolvent set is
non-empty, so for $\lambda \notin \operatorname{ran}c$, $A_\gamma
-\lambda $ is a Fredholm operator (by the ellipticity) with  index 0
(since the index depends continuously on $\lambda $). Then all spectral points outside
$\operatorname{ran}c$ are eigenvalues with finite dimensional
eigenspaces. Since these eigenspaces are linearly independent, there
can only be countably many, so there is at most a countable set of
eigenvalues outside $\operatorname{ran}c$. They can only accumulate at
points of $\operatorname{ran}c$. 
\end{proof}

Consider another boundary condition for $A$,
\begin{equation}
\nu _1u=C\gamma _0u,\label{4.8}
\end{equation}
with $C$ a first-order differential operator, and such that the
system $\{A_0, \nu _1-C\gamma _0\}$ is elliptic, defining the
realisation $\wA_0$. Then   
 $\{A, (\nu _1-C\gamma _0)\pr_1\}$ is likewise elliptic, and we
define $\wA$ to be the realisation of $A$ under the boundary condition
(\ref{4.8}). Again, the essential spectrum equals $\operatorname{ran}c$.

It is well-known that $\wA_0$ satisfies a 1-coerciveness inequality (hence is
lower bounded) if and only if the real part of the principal symbol of
$L=C-P^0_{\gamma _0,\nu _1}$ is $\ge c_0|\xi '|$ with $c_0>0$, and
that the
adjoint then has similar properties (cf.\ e.g.\ \cite{G71},
\cite{G74}). Assuming this, we have that the spectrum lies in a 
half-plane $ \{\operatorname{Re}\lambda \ge \alpha _1\}$, 
and the spectrum is discrete outside $\operatorname{ran}c$.

We next want to discuss $M$-functions for the comparison of the
Dirichlet realisation and the realisation defined by \eqref{4.8}. Let $k$ be a point in $\varrho (A_\gamma )$; then the general 
analysis of Section \ref{sec2} works for the
realisations of $A-k$,  with
reference operator $A_\gamma -k$. So  the
holomorphic families $T^{\lambda }$ and $M_{(\wA-k)}(\lambda )$ are
well-defined relative to this set-up. Note that $c(x)$ is now replaced by
$c_k(x)=c(x)-k$; the essential spectra of $A_\gamma -k$ and $\wA-k$ are
contained in $\operatorname{ran}c_k=\operatorname{ran}c-k$.

Let $\lambda \notin \operatorname{ran}c_k$. The solutions of the
Dirichlet problem with non-homogeneous boundary condition are the
solutions of
\begin{align}
(A_0-k-\lambda )u+av&=0,\nonumber,\\
bu+(c-k-\lambda )v&=0,\nonumber
\\ \gamma _0u&=\varphi.
 \label{4.10}
\end{align}
The second line is solved by $v=b(\lambda +k-c)^{-1}u$, which by
insertion in the first line gives
\begin{align}
 (A_0-k-\lambda +ab(\lambda +k-c)^{-1})u&=0\nonumber
\\ \gamma _0u&=\varphi .
\label{4.11}
\end{align}

When  $\lambda \in\varrho (A_\gamma -k)$, the problem 
(\ref{4.11}) has a unique solution $u=K^{k,\lambda }_\gamma \varphi \in
\operatorname{ker}(A_0-k-\lambda +ab(\lambda +k-c)^{-1})$ for each $\varphi\in H^{-\frac12}
$, and (\ref{4.10}) has the solution $\{u, b(k+\lambda -c)^{-1}u\}\in 
\operatorname{ker}(A_\gamma -k-\lambda )$. The
Dirichlet-to-Neumann operator family for $A-k$ is
\begin{equation}
P^{k,\lambda} =\nu _1\pr_1K^{k,\lambda }_\gamma ,\label{4.12}
\end{equation}
which identifies with the Dirichlet-to-Neumann operator family
$P^{k,\lambda }_{\gamma _0,\nu _1}$ for $A_0-k $.

As above, let $\wA$ be the realisation of $A$ under the boundary condition
(\ref{4.8}). It corresponds as in Section \ref{sec3.2} to
\begin{align*}
L^{\lambda }&=C-P^{k,\lambda }=L+P^{k,0}-P^{k,\lambda }:H^{-\frac12}\to
H^\frac12,\text{ where }\nonumber \\
L&=C-P^{k,0},\quad D(L)=H^\frac32.
\end{align*}

Then there is an $M$-function going in the opposite direction and satisfying 
\begin{equation}
M_L(\lambda )=-(L+P^{k,0}-P^{k,\lambda })^{-1}\label{4.13}
\end{equation}
  when $\lambda \in \varrho (\wA-k)\cap\varrho (A_\gamma -k)$.
With a notation similar to (\ref{3.10}), $M_L(\lambda )$  acts as follows:
\begin{equation}
M_L(\lambda )=P^{k,\lambda }_{\nu _1-C\gamma _0,\gamma _0}.\label{4.14}
\end{equation}

\begin{proposition}
There exists a point belonging to the essential spectrum of
$\wA-k$ (and of $A_\gamma -k$) at which $M_L(\lambda )$ is holomorphic.
\end{proposition}

\begin{proof} 
Let $\lambda _0=c(x_0)-k$, for some $x_0$
where $c_0(x)
\in K\setminus\{0\}$. Then $\lambda
_0$ belongs to $ \sigma _{\operatorname{ess}}(A_\gamma -k)$ and  to $
\sigma _{\operatorname{ess}}(\wA -k)$. We shall show that $M_L(\lambda
)$ can be extended holomorphically across $\lambda _0$ or a neighbouring
point.

In view of (\ref{4.3a}), 
\begin{equation}
a(x)b(x)(\lambda +k-c(x))^{-1}=a(x)b(x)(\lambda +k)^{-1},\text{ for all
$x\in\comega$, all }\lambda \ne -k.\label{4.15}
\end{equation} 

This implies that the problem (\ref{4.11}) takes the form
 \begin{align}
(A_0-k-\lambda +ab(\lambda +k)^{-1})u&=0\nonumber
\\ \gamma _0u&=\varphi ,
 \label{4.16}
\end{align}
for $\lambda \ne -k$, and this is obtained by reduction from the problem
\begin{align}
(A_0-k-\lambda )u+av&=0,\nonumber\\
bu-(\lambda +k)v&=0,\nonumber
\\ \gamma _0u&=\varphi,
 \label{4.17}
\end{align}
a Dirichlet problem for
\begin{equation}
A_1=\begin{pmatrix} A_0-k&a(x)\\b(x)&-k\end{pmatrix}.\label{4.18}
\end{equation}
Similarly, the problem 
\begin{equation}
(A-k-\lambda )w=0,\quad \nu _1u-C\gamma _0u=\psi ,\label{4.19}
\end{equation}
is equivalent with 
\begin{equation}
(A_1-\lambda )w=0,\quad \nu _1u-C\gamma _0u=\psi .\label{4.20}
\end{equation}
So, $M_L(\lambda )$ defined above coincides with the analogous operator for $A_1$: 
\begin{equation}
M_L(\lambda )=P^\lambda _{\nu _1-C\gamma _0,\gamma _0}(A_1).\label{4.21}
\end{equation}
It is holomorphic on $\varrho (\wA_1)$, where $\wA_1$ is the
realisation of $A_1$ defined by the boundary condition (\ref{4.8}).

This detour via $A_1$ gives information about the possible holomorphic
extensions of the  $M_L(\lambda )$-function for $\wA-k$. We infer 
from the general result of
\cite{GG77} that the Dirichlet realisation $A_{1,\gamma }$ of $A_1$,
as well as
the realisation $\wA_1$, have essential spectra equal to
$\{-k\}$. Moreover, their spectra are contained in a 
half-space, and are discrete outside the point $\{-k\}$.

So $\lambda _0=c(x_0)-k$ is either in $\varrho (\wA_1)$ or is one of the
discrete eigenvalues
of $\wA_1$, and in any case there is a
disk  $B(\lambda _0,\delta )$ around it such that $M_L(\lambda )$ is
holomorphic on $B(\lambda _0,\delta )\setminus \{\lambda _0\}$.
Since $c(x_0)-k$ is not the only point in the connected set 
$K-k$, there
will be a  point $x_0'$ such that $c(x'_0)-k\in B(\lambda _0,\delta )\setminus \{\lambda _0\}$.

We can conclude that $M_L(\lambda )$ is holomorphic at $\lambda
_0'=c(x'_0)-k$, but the point belongs to the essential spectrum of
$\wA-k$ (and of $A_\gamma -k$).
\end{proof}

\begin{remark}  \label{Remark4.1}
The hypothesis (\ref{4.3a}) can be replaced by a weaker hypothesis as
follows: Assume that $K$ is a compact connected subset of
$\operatorname{ran}c$ containing more than one point. Let $\omega ,\omega
',\omega ''$ be subsets of $\comega $ with \linebreak $\operatorname{dist}(\omega
,\comega\setminus\omega ')>0$, $\operatorname{dist}(\omega '
,\comega\setminus\omega '')>0$, such that $K\subset c(\omega )$ and
$\operatorname{dist}(K
,c(\comega\setminus\omega '))>0$. {\it
  Assume that $ab$ vanishes on $\omega ''$}. Let $\eta \in C^\infty $
with $\eta =1$ on $\comega \setminus \omega ''$ and $\eta =0$ on
$\omega '$, and set $c'=\eta c$.   
Then, for all $\lambda \notin \operatorname{ran}c'- k$, $ab(\lambda
+k-c)^{-1}=ab(\lambda +k-c')^{-1}$. The problems (\ref{4.10}) and
(\ref{4.19}) can now be replaced by problems where $c$ is replaced by
$c'$ , whose range is disjoint from $K$, so that there will be points in $K-k$ where the $M$-function is holomorphic.
\end{remark}

\subsection{ODE counterexample}\label{sec4.2}

Consider the $2\times 2$ matrix-formed operator
\begin{equation}
A\begin{pmatrix} u\\ v\end{pmatrix}(x)=\begin{pmatrix} -u''(x)&a(x)v'(x) \\b(x)u'(x)&c(x)v(x)\end{pmatrix},\label{4.3b}
\end{equation}
acting on pairs of functions $u,v$ on the interval $[0,1]$ and  $a,b,c\in C^\infty ([0,1])$. 
Its formal adjoint is given by 
\begin{equation}
A'\begin{pmatrix} u\\ v\end{pmatrix}(x)=\begin{pmatrix} -u''(x)&(\overline{b(x)}v(x))' \\(\overline{a(x)}u(x))'&\overline{c(x)}v(x)\end{pmatrix},\label{4.3bb}
\end{equation}
and there is the Green's formula
\[
\left(A\begin{pmatrix} u_1\\ u_2\end{pmatrix}, \begin{pmatrix} v_1\\ v_2\end{pmatrix}\right)-\left( \begin{pmatrix} u_1\\ u_2\end{pmatrix},A'\begin{pmatrix} v_1\\ v_2\end{pmatrix}\right)=(\Gamma _1u, \Gamma '_0v)-(\Gamma _0u ,
\Gamma '_1v)
\] where
\begin{equation}
\Gamma _1u=\begin{pmatrix} -u_1'(1)+a(1)u_2(1)\\  u_1'(0)-a(0)u_2(0) \end{pmatrix},\;\Gamma _0u=\Gamma _0'u= \begin{pmatrix} u_1(1)\\  u_1(0) \end{pmatrix} ,\;
\Gamma '_1v=\begin{pmatrix} -v_1'(1)-b(1)v_2(1)\\  v_1'(0)+b(0)v_2(0) \end{pmatrix}
\end{equation}
and $u,v\in D(\Ama)=H^2(0,1)\times H^1(0,1)$.

\begin{proposition}
We introduce the quantities
\begin{align}\label{4.25}	Q(x,\lambda)=\frac{a(x)\left(\frac{b(x)}{\lambda-c(x)}\right)'}{-1+\frac{a(x)b(x)}{\lambda-c(x)}},\; \alpha(x)=\exp\left(\int_0^xQ(s,\lambda)\ ds\right), \; \beta(x)=\frac{\alpha(x)}{-1+\frac{a(x)b(x)}{\lambda-c(x)}},
\end{align}
and let $y_1$ and $y_2$ be two linearly independent solutions of 
\begin{equation}\label{4.26}
\varphi^{\prime\prime}+\frac{1}{\alpha}\left[\frac{\alpha'^2}{4\alpha}-\frac{\alpha^{\prime\prime}}{2}-\beta\lambda\right]\varphi=0
\end{equation} 
satisfying the initial conditions $y_1(0)=1, y_1'(0)=0, y_2(0)=0, y_2'(0)=1$. Note that $y_1$ and $y_2$ depend on $\lambda$, but we suppress this in the notation. 

{\rm (i)} We have
\[\ker(\Ama-\lambda)=\left\{\begin{pmatrix} c_1y_1 +c_2y_2\\
    \frac{b}{\lambda-c}(c_1y_1' +c_2y_2')\end{pmatrix}\,\Big|\; c_1,c_2\in\C\right\}.
\]

{\rm (ii)} Consider the operator $A_1$, the restriction of $\Ama$ to $\ker \Gamma_1$, i.e.~subject to Neumann boundary conditions.
The matrix of the $M$-function has the entries
\begin{eqnarray} m_{11}=\frac{y_1(1)}{\left[\frac{a(1)b(1)}{\lambda-c(1)}-1\right]y_1'(1)}, \; m_{12}=\frac{1}{\left[1-\frac{a(0)b(0)}{\lambda-c(0)}\right]y_1'(1)},\nonumber  \\
 m_{21}=\frac{1}{\left[\frac{a(1)b(1)}{\lambda-c(1)}-1\right]y_1'(1)}, \;
m_{22}=\frac{y_2'(1)}{\left[1-\frac{a(0)b(0)}{\lambda-c(0)}\right]y_1'(1)}. \label{mfn}
\end{eqnarray}

{\rm (iii)} Assume in addition that the function $b$ vanishes identically on an open interval \mbox{$I\subseteq (0,1)$}. Then there is a second operator $\At$ with different essential spectrum to $A_1$, but giving rise to the same $M$-function.

\end{proposition}

\begin{remark}
Note that if $y_1'(1)=0$, then $y_1$ is a Neumann eigenfunction for (\ref{4.26}), and $\lambda$ is an eigenvalue of the operator $A_1$. Apart from this case, singularities of the $M$-function only occur when the coefficient in (\ref{4.26}) blows up, as terms of the form $\frac{a(x)b(x)}{\lambda-c(x)}-1$ also appear there.
\end{remark}

\begin{proof}
{\rm (i)} 
$(A-\lambda)u=0$ can be written as 
$$-u_1^{\prime\prime}-\lambda u_1+au_2'=0,\; bu_1'+(c-\lambda)u_2=0.$$
Solving the second equation for $u_2$ and substituting this into the first gives
\begin{equation}\label{4.24}
-u_1^{\prime\prime}+a\left(\frac{bu_1'}{\lambda-c}\right)'-\lambda u_1=0.
\end{equation}
Introducing $\alpha$ and $\beta$ as in (\ref{4.25}), the equation (\ref{4.24}) simply becomes $(\alpha u_1')'-\beta\lambda u_1=0$.
Moreover, introducing
\[ \varphi(x)=\exp{\int^x \frac{\alpha'}{2\alpha}}u_1(x),
\]
the equation can be written as
\begin{equation}
\varphi^{\prime\prime}+\frac{1}{\alpha}\left[\frac{\alpha'^2}{4\alpha}-\frac{\alpha^{\prime\prime}}{2}-\beta\lambda\right]\varphi=0
\end{equation}
and the kernel of $\Ama-\lambda$  has the form
\[\ker(\Ama-\lambda)=\left\{\begin{pmatrix} c_1y_1 +c_2y_2\\
    \frac{b}{\lambda-c}(c_1y_1' +c_2y_2')\end{pmatrix}\,\Big|\; c_1,c_2\in\C\right\}.
\]

{\rm (ii)} 
We now calculate the $M$-function for the operator subject to Neumann boundary conditions.
For any $(u_1, u_2)\in \ker(\Ama-\lambda)$ we have
\[\begin{pmatrix} m_{11} & m_{12}\\m_{21} &m_{22} \end{pmatrix}\begin{pmatrix} -u_1'(1)+a(1)u_2(1)\\ u_1'(0)-a(0)u_2(0)\end{pmatrix}=\begin{pmatrix} u_1(1)\\ u_1(0)\end{pmatrix}.
\]
A simple calculation then gives (\ref{mfn}).

%
%

{\rm (iii)} 
We first note that the
determinant of the principal symbol of $A-\lambda$ is given by $\xi^2
(ab+c-\lambda)$ which is zero if $\lambda\in\Ran(ab+c)$. Hence
ellipticity of the system fails for $\lambda\in\Ran(ab+c)$ and by
\cite{GG77},  
the essential spectrum of the operator equals $\Ran(ab+c)$. 
Under the additional assumption that the function $b$ vanishes identically on an open interval \mbox{$I\subseteq (0,1)$}, we have that $\Ran(c\vert_I)$ is contained in the essential spectrum of the operator. Let $\ct$ be a $C^\infty$-function on $[0,1]$ that coincides with $c$ on $[0,1]\setminus I$. This gives rise to another operator $\At$ with
\begin{equation}
\At\begin{pmatrix} u\\ v\end{pmatrix}(x)=\begin{pmatrix} -u''(x)&a(x)v'(x) \\b(x)u'(x)&\ct(x)v(x)\end{pmatrix}.
\end{equation}
Let $\At_1$ be the realisation of $\At$ subject to Neumann boundary conditions. Then $\Ran(\ct\vert_I)$ will lie in $\sigma_{\operatorname{ess}}(\At_1)$. Therefore, in general the essential spectrum of the operators $A_1$ and $\At_1$ will differ. However, for the calculation of the $M$-function, $c$ only appears in terms of the form $\frac{b(x)}{\lambda-c(x)}$  and by our assumptions, 
$$\frac{b(x)}{\lambda-c(x)}=\frac{b(x)}{\lambda-\ct(x)},$$ so the $M$-functions for $A_1$ and $\At_1$ coincide. Thus, we have another example where two operators with differing essential spectra give rise to the same $M$-function.
\end{proof}

%

\bibliographystyle{mn}

\providecommand{\WileyBibTextsc}{}
\let\textsc\WileyBibTextsc
\providecommand{\othercit}{}
\providecommand{\jr}[1]{#1}
\providecommand{\etal}{~et~al.}

\end{document}